%% file: _0LiFEMME.tex
\documentclass[12pt,a4paper,fleqn]{article}
\pdfoutput=1

\usepackage[english]{babel}
\RequirePackage{ucs}
\RequirePackage[utf8x]{inputenc}
\RequirePackage[T1]{fontenc}
\RequirePackage{amsthm,amsmath}
\RequirePackage{amssymb,amstext}
\RequirePackage{amsfonts,stmaryrd}
\RequirePackage{amsbsy} 
\RequirePackage{mathrsfs}
\RequirePackage{bm}
\RequirePackage{bbm}
\RequirePackage{aliascnt,ifthen}
\RequirePackage{graphicx}
\RequirePackage{color,calc}
\RequirePackage{times}
\RequirePackage{booktabs}
\RequirePackage{dsfont} 
\usepackage{url}
\usepackage{verbatim}
\usepackage{subcaption}
\usepackage[pdfpagemode=UseOutlines ,plainpages=false,
hyperindex=true,colorlinks=true]{hyperref}
	\makeatletter
	\Hy@breaklinkstrue
	\makeatother

\definecolor{darkred}{rgb}{0.6,0,0.1}
\definecolor{darkgreen}{rgb}{0,0.5,0}
\definecolor{darkblue}{rgb}{0,0,0.5}

\hypersetup{colorlinks ,linkcolor=gray ,filecolor=darkgreen, urlcolor=darkblue ,citecolor=black}

\RequirePackage{paralist}
\RequirePackage{relsize}
\usepackage{enumitem}

\usepackage{natbib}
\usepackage{hypernat}
\renewcommand{\cite}{\citet}
\usepackage{LIFEMME-abrev-package}

\definecolor{dgreen}{rgb}{0,0.5,0}
\definecolor{dblue}{rgb}{0,0,0.5}
\definecolor{dred}{rgb}{0.6,0.0,0.1}
\definecolor{dgold}{rgb}{0.5,0.3,0.0}
\definecolor{dvio}{rgb}{0.6,0.3,0.5}
\definecolor{gray}{rgb}{0.5,0.5,0.5}
\definecolor{dbraun}{rgb}{.5,0.2,0}

\newcommand{\dgrau}{\color{gray}}

\newcommand{\colre}{dred}
\newcommand{\colas}{dblue}
\newcommand{\colrem}{dgold}
\newcommand{\colil}{dgreen}

\newtheoremstyle{styre}% name
  {1.1\topsep}%      Space above
  {\topsep}%      Space below
  {\normalfont\itshape}%         Body font
  {}%         Indent amount (empty = no indent, \parindent = para indent)
  {\color{\colre}}% Thm head font
  {.}%        Punctuation after thm head
  {.5em}%     Space after thm head: " " = normal interword space;
  {\thmname{\textbf{#1}\xspace}{\xspace\dgrau\thmnumber{#2}}\thmnote{\xspace\textit{\small(#3)}}}%         Thm head spec (can be left empty, meaning `normal')

\newtheoremstyle{styas}% name
  {1.1\topsep}%      Space above
  {\topsep}%      Space below
  {\normalfont\itshape}%         Body font
  {}%         Indent amount (empty = no indent, \parindent = para indent)
  {\color{\colas}}% Thm head font
  {.}%        Punctuation after thm head
  {.5em}%     Space after thm head: " " = normal interword space;
  {}%         Thm head spec (can be left empty, meaning `normal')

\newtheoremstyle{styrem}% name
  {1.1\topsep}%      Space above
  {\topsep}%      Space below
  {\normalfont\itshape}%         Body font
  {}%         Indent amount (empty = no indent, \parindent = para indent)
  {\color{\colrem}}% Thm head font
  {.}%        Punctuation after thm head
  {.5em}%     Space after thm head: " " = normal interword space;
  {\thmname{\textbf{#1}\xspace}{\xspace\dgrau\thmnumber{#2}}\thmnote{\xspace\textit{\small(#3)}}}%{}%         Thm head spec (can be left empty, meaning `normal')

\newtheoremstyle{styil}% name
  {1.1\topsep}%      Space above
  {\topsep}%      Space below
  {\normalfont\rmfamily}%         Body font
  {}%         Indent amount (empty = no indent, \parindent = para indent)
  {\color{\colil}}% Thm head font
  {.}%        Punctuation after thm head
  {.5em}%     Space after thm head: " " = normal interword space;
  {\thmname{\textbf{#1}\xspace}{\xspace\dgrau\thmnumber{#2}}\thmnote{\xspace\textit{\small(#3)}}}%         Thm head spec (can be left empty, meaning `normal')

\newtheoremstyle{stypro}%
	{0.5\topsep}%space above
	{1.1\topsep}%space below
	{\upshape}%		body font
	{}%				indent amount
	{}%	theorem head font
	{.}%			punctuation after theorem head
	{.5em}%			space after theorem head
	{\thmnote{\textit{#3}}}%         Thm head spec (can be left empty, meaning `normal')

\theoremstyle{styre}\newtheorem{pr}{Proposition}[section]
\newaliascnt{co}{pr}
\theoremstyle{styre}\newtheorem{co}[co]{Corollary}
\aliascntresetthe{co}
\newaliascnt{thm}{pr}
\theoremstyle{styre}\newtheorem{thm}[thm]{Theorem}
\aliascntresetthe{thm}
\newaliascnt{lem}{pr}
\theoremstyle{styre}\newtheorem{lem}[lem]{Lemma}
\aliascntresetthe{lem}
\newaliascnt{rem}{pr}
\theoremstyle{styrem}\newtheorem{rem}[rem]{Remark}
\aliascntresetthe{rem}
\newaliascnt{ex}{pr}
\theoremstyle{styil}\newtheorem{ex}[ex]{Example}
\aliascntresetthe{ex}
\newaliascnt{il}{pr}
\theoremstyle{styil}\newtheorem{il}[il]{Illustration}
\aliascntresetthe{il}
\theoremstyle{styas}

\theoremstyle{styas}

\theoremstyle{stypro}\newtheorem*{pro}{}

\newcommand{\proEnd}{{\scriptsize\textcolor{\colre}{\qed}}}
\newcommand{\ilEnd}{{\scriptsize\textcolor{\colil}{\qed}}}

\usepackage{cleveref}
\crefname{pr}{\color{\colre}Proposition}{\color{\colre}Propositions}
\crefname{co}{\color{\colre}Corollary}{\color{\colre}Corollaries}
\crefname{thm}{\color{\colre}Theorem}{\color{\colre}Theorems}
\crefname{lem}{\color{\colre}Lemma}{\color{\colre}Lemmata}
\crefname{ass}{\color{\colas}Assumption}{\color{\colas}Assumptions}
\crefname{de}{\color{\colas}Definition}{\color{\colas}Definitions}
\crefname{rem}{\color{\colrem}Remark}{\color{\colrm}Remarks}
\crefname{il}{\color{\colil}Illustration}{\color{\colil}Illustrations}
\crefname{ex}{\color{\colil}Example}{\color{\colil}Examples}
\usepackage{environ}
\NewEnviron{te}{%
{\BODY}%
}

\numberwithin{equation}{section} 

\makeatletter
\newcommand{\mylabel}[2]{#2\def\@currentlabel{#2}\label{#1}}
\makeatother

\newcommand{\setListe}[5][3ex]{\setlength{\itemsep}{#2}\setlength{\topsep}{#3}\setlength{\leftmargin}{#4}\setlength{\rightmargin}{#5}\setlength{\labelwidth}{#1}}
\newcommand{\setListeStandard}{\setListe{0ex}{.5ex}{4ex}{0ex}}
\newcounter{ListeN}
\renewcommand{\theListeN}{(\roman{ListeN})}

\newenvironment{resListeN}[1][~]%
{\setcounter{ListeN}{0}\renewcommand{\theListeN}{\normalfont\rmfamily\color{\colre}(\roman{ListeN})}\begin{list}{\theListeN}%
{\usecounter{ListeN}\setListeStandard #1}}
{\end{list}}

\newenvironment{exListeN}[1][~]%
{\setcounter{ListeN}{0}\renewcommand{\theListeN}{\normalfont\rmfamily\color{\colil}(\roman{ListeN})}\begin{list}{\theListeN}%
{\usecounter{ListeN}\setListeStandard #1}}
{\end{list}}

\newenvironment{Liste}[1][~]%
{\begin{list}{}%
{\setListeStandard #1}}
{\end{list}}

\makeatletter% 
\def\@fnsymbol#1{\ensuremath{\ifcase#1\or * \or ** \or 2 \or 3 \or  *\or  \star \or 4\or  , \or 
g\or h\or i\else\@ctrerr\fi}}% 
\makeatother% 

\author{\begin{minipage}{.45\textwidth}\center{\sc Sergio Brenner Miguel}\;\thanks{Institut f\"ur Angewandte
    Mathematik, M$\Lambda$THEM$\Lambda$TIKON, Im Neuenheimer Feld 205,
  D-69120 Heidelberg, Germany, e-mail:
  \url{{brennermiguel|johannes}@math.uni-heidelberg.de}}\\\small Ruprecht-Karls-Universit\"{a}t Heidelberg\\\null
\end{minipage} \and \begin{minipage}{.45\textwidth}\center{\sc 
  Fabienne Comte} \thanks{CNRS, MAP5 UMR 8145,
  F-75006 Paris, France, e-mail:
  \url{fabienne.comte@parisdescartes.fr}}\\\small Universit\'e de Paris\\\null\end{minipage}\and\begin{minipage}{.45\textwidth}\center{\sc 
  Jan Johannes}$\;^*$\\\small Ruprecht-Karls-Universit\"{a}t Heidelberg\\\null\end{minipage}}
\date{\today} 
\title{Linear functional estimation under multiplicative measurement errors} 
\begin{document} 
\maketitle 
% --------------------------------------------------------------------
% <<Abstract>>
% --------------------------------------------------------------------
\begin{abstract}
  We study the non-parametric estimation of the value $\vartheta(f)$
  of a linear functional evaluated at an unknown density function $f$ with support on $\pRz$ based on an i.i.d. sample with multiplicative measurement errors. 
  The proposed estimation procedure combines the estimation of the Mellin
  transform of the density $f$ and  a regularisation of the inverse of
  the Mellin transform by a spectral cut-off. In order to bound the
  mean squared error we distinguish several scenarios characterised
  through different decays of the upcoming Mellin transforms and the
  smoothnes of the linear functional. In fact, we identify scenarios, where a non-trivial choice of the upcoming tuning parameter is necessary and propose a data-driven choice based on a Goldenshluger-Lepski method.
 Additionally,  we show minimax-optimality over \textit{Mellin-Sobolev spaces} of the
 estimator.
\end{abstract} 
{\footnotesize
\begin{tabbing} 
\> Mellin-Sobolev space,  minimax theory,inverse problem, adaptation\\[.2ex] 
\end{tabbing}}

% --------------------------------------------------------------------
% <<Content>>
% --------------------------------------------------------------------
\input{_1Intro}
\input{_2DataDriven}
\input{_3Minimax}
% --------------------------------------------------------------------
% <<Appendix>>
% --------------------------------------------------------------------
\appendix
\setcounter{subsection}{0}
\section*{Appendix}
\numberwithin{equation}{subsection}  
\renewcommand{\thesubsection}{\Alph{subsection}}
\renewcommand{\theco}{\Alph{subsection}.\arabic{co}}
\numberwithin{co}{subsection}
\renewcommand{\thelem}{\Alph{subsection}.\arabic{lem}}
\numberwithin{lem}{subsection}
\renewcommand{\therem}{\Alph{subsection}.\arabic{rem}}
\numberwithin{rem}{subsection}
\renewcommand{\thepr}{\Alph{subsection}.\arabic{pr}}
\numberwithin{pr}{subsection}

\input{_app2.tex} 
\input{_app3.tex}
% --------------------------------------------------------------------
% <<BibFile>>
% --------------------------------------------------------------------
\bibliography{LIFEMME} 
\end{document}

%% file: _1Intro.tex
%======================================================================================================================
%                                                                 
% Title: Introduction
% Author: Sergio Brenner Miguel, Fabienne Comte,  Jan JOHANNES 
% 
% Date: %%ts latex start%%[2021-11-29 Mon 15:49]%%ts latex end%%
%
% ======================================================================================================================
% ====================================================================
% Section <<Intro>>
% ====================================================================
\section{Introduction}\label{i}
% ....................................................................
% Motivation, Objectiv
% ....................................................................
\begin{te} 
 	In this paper we are interested in estimating the value $\vartheta(f)$ of a linear functional evaluated at an unknown density $f: \pRz \rightarrow \pRz$ of a positive random variable $X$, when $Y=XU$ for some  multiplicative positive error term $U$ is only observable.  
 	We assume that $X$ and $U$ are independent of each other and
        that $U$ has a known density $g:\pRz \rightarrow \pRz$. In a
        multiplicative measurement errors model the density of
        $f_Y:\pRz \rightarrow \pRz$ of the observable $Y$ is thus given by
 	\begin{align*}
 	f_Y(y)=[f*g](y)= \int_0^{\infty} f(x)g(y/x)x^{-1}dx, \quad y\in \pRz,
 	\end{align*}
        such that $*$ denotes multiplicative convolution. Therefore,
        the estimation of $f$ and hence $\vartheta(f)$ using an
        i.i.d. sample $Y_1, \dots, Y_n$ from $f_Y$ is called a
        multiplicative deconvolution problem,
        which is an inverse problem. \\
        \cite{Vardi1989} and \cite{VardiZhang1992} introduce and study
        intensively \textit{multiplicative censoring}, which
        corresponds to the particular multiplicative deconvolution
        problem with multiplicative error $U$ uniformly distributed on
        $[0,1]$.  \textit{Multiplicative censoring} is a common
        challenge in survival analysis as explained and motivated in
        \cite{Van-EsKlaassenOudshoorn2000}. The estimation of the
        cumulative distribution function of $X$ is considered in
        \cite{VardiZhang1992} and \cite{AsgharianWolfson2005}. Series
        expansion methods are studied in \cite{AndersenHansen2001}
        treating the model as an inverse problem. The density
        estimation in a multiplicative censoring model is considered
        in \cite{BrunelComteGenon-Catalot2016} using a kernel
        estimator and a convolution power kernel estimator. Assuming a
        uniform error distribution on an interval
        $[1-\alpha, 1+\alpha]$ for $\alpha\in (0,1)$
        \cite{ComteDion2016} analyze a projection density estimator
        with respect to the Laguerre basis.
        \cite{BelomestnyComteGenon-Catalot2016} study a
 	beta-distributed error $U$.\\
 	The multiplicative measurement error model covers all those
        three variations of multiplicative censoring. It was
        considered by \cite{BelomestnyGoldenshlugerothers2020} for the
        point-wise density estimation.  The key to the analysis of
        multiplicative deconvolution is the multiplication theorem,
        which for a density $f_Y=f *g$ and their Mellin transforms
        $\Melop[f_Y]$, $\Melop[f]$ and $\Melop[g]$ (defined below)
        states $\Melop[f_Y]=\Melop[f]\Melop[g]$. Exploiting the
        multiplication theorem
        \cite{BelomestnyGoldenshlugerothers2020} introduce a kernel
        density estimator of $f$ allowing more generally $X$ and $U$
        to take also negative values. Moreover, they point out that
        transforming the data by applying the logarithm is a special
        case of their estimation strategy. Note that by applying the
        logarithm the model $Y=XU$ writes $\log(Y)=\log(X)+\log(U)$,
        and hence multiplicative convolution becomes (additive)
        convolution for the $\log$-transformed data. As a consequence,
        first the density of $\log(X)$ is eventually estimated
        employing usual strategies for non-parametric (additive)
        deconvolution problems (see for example \cite{Meister2009})
        and then secondly transformed back to an estimator of
        $f$. Thereby, regularity conditions commonly used in
        (additive) deconvolution problems are imposed on the density
        of $\log(X)$, which however is difficult to interpret as
        regularity conditions on the density of $f$.  Furthermore, the
        analysis of a global risk of an estimator $f$ using this naive
        approach is challenging as
        \cite{ComteDion2016}	point out.\\
 	The global estimation of the density under multiplicative
        measurement errors is considered in
        \cite{Brenner-MiguelComteJohannes2020} using the Mellin
        transform and a spectral cut-off regularization of its inverse
        to define an estimator for the unknown density $f$.
        \cite{Brenner-Miguel2021} studies the global density
        estimation under multiplicative measurement errors for
        multivariate random variables while the global estimation of
        the survival function can be found in
        \cite{Brenner-MiguelPhandoidaen2021}. In this paper we
        estimate the value $\vartheta(f)$ of a known linear functional
        of the unknown density $f$ plugging in the estimator of $f$
        proposed by \cite{Brenner-MiguelComteJohannes2020}.  In
        additive deconvolution linear functional estimation has been
        studied for instance by \cite{ButuceaComte2009},
        \cite{Mabon2016} and \cite{Pensky2017} to mention only a few.
        In the literature, the most studied examples for estimating
        linear functionals is point-wise estimation of the unknown
        density $f$, the survival function, cumulative distribution
        function (c.d.f.) or the Laplace transform of $f$. These
        examples are particular cases of our general setting.  More
        precisely, we show below, that in each of those examples the quantity of
        interest can be written as linear functional in the form
 	\begin{align*}
 	\vartheta(f):= \frac{1}{2\pi}\int_{-\infty}^{\infty} \Psi(-t) \Mela{f}{}(t)dt,
 	\end{align*}
 	where $\Psi:\Rz \rightarrow \Cz$ is a known function and
        $\Mela{f}{}$ denotes the Mellin transform of $f$.\\
        Exploiting properties of the Mellin transform we characterize
        the underlying inverse problem and natural regularity
        conditions which borrow ideas from the inverse problems
        community (see e.g. \cite{EnglHanke-BourgeoisNeubauer2000}).
        More precisely, we identify conditions on the decay of the
        Mellin transform of $f$ and $g$ and of the function $\Psi$ to
        ensure that our estimator is well-defined. We illustrate
        those conditions by different scenarios. The proposed
        estimator, however, involves a tuning parameter and we specify
        when this parameter has to be chosen non-trivially. For that
        case, we propose a data-driven choice of the tuning parameter
        inspired by the work of \cite{GoldenshlugerLepski2011} who
        consider data-driven bandwidth selection in kernel density
        estimation. We establish an oracle inequality for the plug-in
        spectral cut-off estimator under fairly mild assumptions on
        the error density $g$.  Moreover we show that uniformly over
        \textit{Mellin-Sobolev spaces} the proposed estimator is
        minimax-optimal.\\
        The paper is organized in the following way: in \cref{dd} we
        develop the data-driven plug-in estimator and introduce our basic
        assumptions. We state  an oracle type upper bound for the
        mean squared error of the plug-in spectral cut-off estimator with
        fully-data driven choice of the tuning parameter. In \cref{mm}
        we state a maximal upper bound over Mellin-Sobolev spaces 
        mean squared error of the spectral cut-off estimator for
        the plug-in spectral cut-off estimator with optimal tuning
        parameter realising a squared-bias-variance trade-off
        and  lower bounds for the
        point-wise estimation of the unknown density $f$, the survival
        function and the c.d.f. The proofs can be found in the appendix.
\end{te}

% ....................................................................
% Organisation
% ....................................................................

%%% Local Variables:
%%% mode: latex
%%% TeX-master: "_0DANDER"
%%% End:

%% file: _2DataDriven.tex
%======================================================================================================================
%                                                                 
% Title: Minimax theory
% Author: Sergio Brenner Miguel, Fabienne Comte, Jan JOHANNES
% 
% Date: %%ts latex start%%[2021-11-29 Mon 15:48]%%ts latex end%%
%
% ======================================================================================================================
\section{Data-driven estimation}\label{dd}
We begin this section by introducing the Mellin transform and collecting some of its properties.  
We define for a measurable weight function
$\omega:\Rz \rightarrow \pRz$, a constant $p\in \pRz$  and a
measurable set $\Omega\subseteq \Rz$ the 
weighted $\mathbb L^p_{\Omega}(\omega)$-norm  of any measurable
function $h:\Omega\rightarrow \Cz$ by $\|h\|_{\Lz^p_{\Omega}(\omega)}^p := \int_{\Omega}
|h(x)|^p\omega(x)dx $. Denote by
$\Lz^p_{\Omega}(\omega)$ the set of all measurable functions from $\Omega$ to $\mathbb C$ with
finite $\|\, .\,\|_{\Lz^p_{\Omega}(\omega)}$-norm. In the case $p=2$ let $\langle h_1, h_2
\rangle_{\Lz^2_{\Omega}(\omega)} := \int_{\Omega}  h_1(x)
\overline{h_2(x)}\omega(x)dx$ for $h_1, h_2\in \Lz^2_{\Omega}(\omega)$
be the corresponding weighted scalar product. Using a slight abuse of
notation $x^{a}$ with $a\in \mathbb R$ denotes the weight function
$x\mapsto x^{a}$, and we write $\|\cdot \|_{x^{a}}:=
\|\cdot\|_{\Lz^2_{\Rz}(x^{a})}$, respectively $\langle \cdot, \cdot
\rangle_{x^{a}}:=\langle \cdot,
\cdot\rangle_{\Lz^2_{\pRz}(x^{2c-1})}$. Further we use the
abbreviation $\Lz^p_{\mathbb R}=\Lz^p_{\mathbb R}(\omega)$ for the
unweighted $\Lz^p_{\Rz}$ space with $\omega(x)=1$ for all $x\in
\mathbb R$. For a measurable function $h:\Rz \rightarrow \Cz$ let us denote by $\|h\|_{\infty, \omega}$ the essential supremum of the function $x\mapsto h(x)\omega(x)$.

\paragraph{Mellin transform}
Let $c\in \Rz$. For two functions $h_1,h_2\in \mathbb L^1_{\pRz}(
x^{c-1})$ and any $y\in \Rz$ we have $\int_0^{\infty} |h_1(x)h_2(y/x) x^{-1}| dx<\infty$ which allows us to define their  multiplicative convolution $h_1*h_2:\mathbb R\rightarrow \mathbb C$  through
\begin{align}
(h_1*h_2)(y)=\int_0^{\infty} h_1(y/x) h_2(x) x^{-1} dx, \quad y\in \mathbb R.
\end{align}
For a proof sketch of $h_1*h_2 \in \Lz^1_{\pRz}(x^{c-1})$ and the
following properties we refer to \cite{Brenner-Miguel2021}. If in
addition $h_1\in \mathbb L^2_{\pRz}( x^{2c-1})$ (respectively $h_2\in
\Lz^2_{\pRz}( x^{2c-1})$) then $h_1*h_2 \in \mathbb
L^2_{\pRz}(x^{2c-1})$, too. For $h\in \Lz^1_{\pRz}(x^{c-1})$ we define its \textit{Mellin transform} $\mathcal M_c[h]:\Rz \rightarrow \mathbb C$ at the development point $c\in \mathbb R$ by
\begin{align}
\mathcal M_c[h](t):= \int_0^{\infty} x^{c-1+it} h(x)dx, \quad t\in \mathbb R.
\end{align}
One key property of the Mellin transform, which makes it so appealing
for multiplicative deconvolution problems, is the 
multiplication theorem, which for $h_1, h_2\in \mathbb L^1_{\pRz}(x^{c-1})$ states
\begin{align}\label{eq:mel:dec}
\mathcal M_c[h_1*h_2](t)=\mathcal M_c[h_1](t) \mathcal M_c[h_2](t), \quad t\in \Rz.
\end{align}
Making use of the Fourier transform, the domain of definition of the Mellin transform can be extended to $\mathbb L^2_{\pRz}(x^{2c-1})$. Therefore, let $\varphi:\mathbb R\rightarrow \pRz$, with $x\mapsto \exp(-2\pi x)$ and denote by $\varphi^{-1}:\pRz \rightarrow \Rz$ its inverse. Note that the diffeomorphisms $\varphi, \varphi^{-1}$ map Lebesgue null sets on Lebesgue null sets. Consequently, the map $\Phi_c: \Lz^2_{\pRz}(x^{2c-1}) \rightarrow \Lz^2_{\Rz}$, with $h\mapsto \varphi^c \cdot(h\circ \varphi)$ is a well-defined isomorphism and denote by $\Phi_c^{-1}: \Lz^2_{\Rz} \rightarrow \Lz^2_{\pRz}(x^{2c-1})$ its inverse. For $h\in \Lz^2 _{\pRz}(x^{2c-1})$ the Mellin transform $\mathcal M_c[h]: \mathbb R\rightarrow \mathbb C$ developed in $c\in \mathbb R$ is defined through
\begin{align*}\label{eq:mel;def}
	\Mela{h}{c}(t):= (2\pi) \mathcal F[\Phi_c[h]](t) \quad \text{for any } t\in \Rz.
\end{align*} 
 Here,  $\mathcal F: \Lz^2_{\Rz}\rightarrow \Lz^2_{\Rz}$ with $
 H\mapsto (t\mapsto \mathcal F[H](t):=\lim_{k\rightarrow
   \infty}\int_{-k}^k \exp(-2\pi i t x) H(x) dt)$ denotes  the
 Plancherel-Fourier transform where the limit is understood in a
 $\Lz^2_{\Rz}$ convergence sense. Due to this definition several
 properties of the Mellin transform can be deduced from the well-known
 Fourier theory. In particular for any $h\in \Lz^1_{\pRz}(x^{c-1})
 \cap \LpA(x^{2c-1})$ we have
\begin{align}
	\Mela h c(t) =\int_0^{\infty} x^{c-1+it} h(x)dx \quad \text{for any } t\in \Rz,
\end{align}
which coincides with the common definition of a  Mellin transform given in \cite{ParisKaminski2001}. 
\begin{ex}\label{ex:mel:well}
		Now let us give a few examples of Mellin transforms of commonly considered distribution families.
		\begin{exListeN}
			\item \textit{Beta Distribution} admits a
                          density  $g_b(x):= \1_{(0,1)}(x)
                          b(1-x)^{b-1}$ for a $b\in \Nz$ and $x\in
                          \pRz$. Then, we have $g_b\in
                          \Lz^2_{\Rz^+}(x^{2c-1})\cap
                          \Lz^1_{\Rz^+}(x^{c-1})$  for any $c>0$ and 			\begin{align*}
			\Mela{g_b}{c}(t) =\prod_{j=1}^{b} \frac{j}{c-1+j+it}, \quad t\in \Rz.
			\end{align*} 
			\item \textit{Scaled Log-Gamma Distribution}
                          given by its density $g_{\mu, a, \lambda}(x)=\frac{\exp(\lambda \mu)}{\Gamma(a)} x^{-\lambda-1} (\log(x)-\mu)^{a-1}\1_{(e^{\mu}, \infty)}(x)$ for $a,\lambda, x \in \Rz^+$ and $\mu \in \Rz$. Then, for $c<\lambda+1$ hold $g_{\mu,a,\lambda }\in \Lz^2_{\Rz^+}(x^{2c-1})\cap \Lz^1_{\Rz^+}(x^{c-1})$ and
			\begin{align*}
			\Mela{g_{\mu, a, \lambda}}{c}(t)= \exp(\mu(c-1+it)) (\lambda-c+1-it)^{-a}, \quad t\in \Rz.
			\end{align*}
			Note that $g_{\mu,1,\lambda}$ is the density
                        of a Pareto distribution with parameter
                        $e^{\mu}$ and $\lambda$ and $g_{0, a, \lambda}$ is the density of a Log-Gamma distribution.
			\item \textit{Gamma Distribution}  admits a
                          density $g_d(x) = \frac{x^{d-1}}{\Gamma(d)}
                          \exp(-x) \1_{\Rz^+}(x)$ for $d,x\in
                          \Rz^+$. Then, for $c>-d+1$ we have  $g_d\in \Lz^2_{\Rz^+}(x^{2c-1})\cap \Lz^1_{\Rz^+}(x^{c-1})$  and 
			\begin{align*}
			\mathcal M_c[g_d](t)= \frac{\Gamma(c+d-1+it)}{\Gamma(d)}, \quad t\in \Rz.
			\end{align*}
				\item \textit{Weibull Distribution}
                                  admits a density $g_m(x) = m x^{m-1}
                                  \exp(-x^m) \1_{\pRz}(x)$ for $m,x\in
                                  \pRz$. For $c>-m+1$, $\mathcal M_c[g_m]$ is well-defined  and 
			\begin{align*}
			\mathcal M_c[g_m](t)= \frac{(c-1+it)}{m}\Gamma\left(\frac{c-1+it}{m}\right), \quad t\in \Rz.
			\end{align*}
			\item \textit{Lognormal Distribution} admits
                          a density
                          $g_{\mu,\lambda}(x)=\frac{1}{\sqrt{2\pi}\lambda
                            x}
                          \exp(-(\log(x)-\mu)^2/2\lambda^2)\1_{\pRz}(x)$
                          for $\lambda, x \in \pRz$ and $\mu \in
                          \Rz$. $\Mela{g_{\mu,\lambda}}{c}$ is
                          well-defined for any $c\in \Rz$ and it holds
			\begin{align*}
			\Mela{g_{\mu,\lambda}}{c}(t)= \exp(\mu(c-1+it))\exp\left(\frac{\sigma^2(c-1+it)^2}{2}\right), \quad t\in \Rz.
			\end{align*}
		\end{exListeN}
\end{ex}

By construction the operator $\mathcal M_c:
\LpA(x^{2c-1}) \rightarrow \Lz^2_{\Rz}$  is an isomorphism and we denote by $\mathcal M_c^{-1}: \Lz^2_{\Rz} \rightarrow \LpA(x^{2c-1})$ its inverse. If $H\in \Lz^1_{\Rz}\cap  \Lz^2_{\Rz}$ then the inverse Mellin transform is explicitly expressed through
\begin{align}\label{eq:Mel:inv}
\Melop_{c}^{-1}[H](x)= \frac{1}{2\pi } \int_{-\infty}^{\infty} x^{-c-it} H(t) dt, \quad \text{ for any } x\in \pRz.
\end{align} 
Furthermore,  a Plancherel-type equation holds for the Mellin
transform. Precisely,  for all $ h_1, h_2 \in \LpA(x^{2c-1})$ we have
\begin{align}\label{eq:Mel:plan}
\hspace*{-0.5cm}\langle h_1, h_2 \rangle_{x^{2c-1}} = (2\pi)^{-1} \langle \Mela{h_1}{c}, \Mela{h_2}{c} \rangle_{\Lz^2_{\Rz}} \text{ and  } \| h_1\|_{x^{2c-1}}^2=(2\pi)^{-1} \|\Mela{h_1}{c}\|_{\Lz^2_{\Rz}}^2.
\end{align} 
\paragraph{Linear functional}
In the following paragraph we introduce the linear functional,
motivate it through a collection of examples and determine sufficient
conditions to ensure that the considered objects are well-defined. We
then define an estimator based on the empirical Mellin transform and
the multiplication theorem for Mellin transforms.
Let $c\in \Rz$ and $f\in \LpA(x^{2c-1})$. In the sequel we are interested in estimating the linear functional 
\begin{align}\label{eq:lin:fun:1}
\vartheta(f) := \frac{1}{2\pi} \int_{-\infty}^{\infty} \Psi(-t) \Mela{f}{c}(t) dt
\end{align}
for a function $\Psi: \Rz \rightarrow \Cz$ with $\overline{\Psi(t)}=\Psi(-t)$ for any $t\in \Rz$ and such that $\Psi\mathcal M_c[f]\in \mathbb L^1_{\mathbb R}$. The slattern is fulfilled, if $\Psi\in \Lz^2_{\Rz}$. Nevertheless a more detailed analysis of the decay of $\Mela{f}{c}$ and $\Psi$ allows to ensure the integrability in a less restrictive situation. 
Before we present an estimator for $\vartheta(f)$ let us briefly
illustrate our general approach by typical examples.
\begin{il}\label{il:lin:fun:example}
We study in the sequel point-wise estimation at a given point $x_o\in \pRz$ in the following four examples.
			\begin{exListeN}
		\item\label{il:lin:fun:example:i}\textit{Density:}  Introducing the evaluation
                  $f(x_o)$  of $f$ at the point $x_o$ if $\mathcal
                  M_c[f] \in \Lz^1_{\Rz}$ then we have $f(x_o)=\mathcal
                  M_c^{-1}[\mathcal M_c[f]](x_o)=\vartheta(f)$ with
                  $\Psi(t):=x_o^{-c+it}$, $t\in \Rz$, satisfying
                  $\overline{\Psi(t)}=\Psi(-t)$ for all $t\in \Rz$.
		\item\label{il:lin:fun:example:ii}\textit{Cumulative distribution
                    function:} Considering
                   the evaluation $F(x_o)=\int_0^{x_o} f(x)dx$ 
                  of the c.d.f. $F$ at the point $x_o$ define for
                  $c<1$ the function $\psi(x):=x^{1-2c} \mathds 1_{(0,
                    x_o)}(x)$, $x\in \pRz$, which belongs to
                  $\Lz_{\pRz}^1(x^{c-1}) \cap \LpA(x^{2c-1})$. Setting
		\begin{align*}
		\Psi(t):= \Mela{\psi}{c}(t)= \int_0^{x_o} x^{-c+it} dx = (1-c +it)^{-1} x_o^{1-c+it}
		\end{align*}
		we get $\vartheta(f)=F(x_o)$ by an application of the Plancherel equality.
              \item\label{il:lin:fun:example:iii}\textit{Survival function:}
                Introducing the evaluation $S(x_o)=\int_{x_o}^{\infty}
                f(x)dx$  of the survival function $S$ at the point
                $x_o$ define the function $\psi(x):=x^{1-2c}
                \mathds 1_{(x_o, \infty)}(x)$, $x\in \pRz$,  which for $c>1$ belongs to  $\Lz_{\pRz}^1(x^{c-1}) \cap
                \LpA(x^{2c-1})$. Setting
		\begin{align*}
		\Psi(t):= \Mela{\psi}{c}(t)= \int_{x_o}^{\infty} x^{-c+it} dx = -(1-c +it)^{-1} x_o^{1-c+it}
		\end{align*}
		we get $\vartheta(f)=S(x_o)$ by an application of the Plancherel equality.
              \item\label{il:lin:fun:example:iv}\textit{Laplace transform:}
                Given the evaluation $L(x_o)=\int_{0}^{\infty}
                \exp(-x_ox)f(x)dx$ of the Laplace transform $L$ at the
                point $x_o$ define for $c<1$ the function
                $\psi(x):=x^{1-2c} \exp(-t_ox)$, $x\in \pRz$, which
                belongs to  $\Lz_{\pRz}^1(x^{c-1}) \cap
                \LpA(x^{2c-1})$. Setting
		\begin{align*}
		\Psi(t):= \Mela{\psi}{c}(t)= \int_0^{\infty} x^{-c+it} \exp(-x_o x)dx = x_o^{c-1-it}\Gamma(1-c+it)
		\end{align*}
		we get $\vartheta(f)=L(x_o)$ by an application of the Plancherel equality.
	\end{exListeN}
\end{il}
It is worth stressing out that in all four examples introduced in
\cref{il:lin:fun:example}, the quantity of interest  is independent of
the choice of the model parameter $c\in \Rz$.  However, the conditions on
$c\in \Rz$  given \cref{il:lin:fun:example} and the assumption $f\in \Lz_{\pRz}^2(x^{2c-1})$
ensure  that the
representation $\vartheta(f)$  is well-defined, and hence are
essential for our estimation strategy. Consequently, we present the upcoming theory for almost arbitrary choices of $c\in \Rz$.
\begin{rem}
  Consider \cref{{il:lin:fun:example}}. Since $S=1-F$ there is an
  elementary connection between the estimation of the survival
  function and the estimation of the c.d.f.. For example, we
  eventually deduce from a c.d.f. estimator $\widehat F(x_o)$ a
  survival function estimator $\widehat S(x_o)$ through
  $\widehat S(x_o):= 1- \widehat F(x_o)$ with same risk, that is
  $\E_{f_Y}((\widehat S(x_o)- S(x_o))^2) = \E_{f_Y}((\widehat Fx_o)-
  F(x_o))^2)$. Thus we can define for any $c\neq 1$ a survival
  function (respectively c.d.f.) estimator using the results of (ii)
  and (iii) in \cref{il:lin:fun:example}.
\end{rem}
\paragraph{Estimation strategy}
To define an estimator of the quantity $\vartheta(f)$ we make use of
the multiplication theorem \eqref{eq:mel:dec}  as it
is common for deconvolution problems. To do so, let  $f\in \Lz_{\pRz}^2(x^{2c-1})\cap \Lz_{\pRz}^1(x^{c-1})$ and $g\in \Lz_{\pRz}^1(x^{c-1})$ then we deduce $\Mela{f_Y}{c}(t)=\Mela{f}{c}(t)\Mela{g}{c}(t)$ for all $t\in \Rz$ by application of the multiplication theorem. Under the mild assumption that $\Mela{g}{c}(t) \neq 0$ for all $t\in \Rz$ we conclude that $\Mela{f}{c}(t)=\Mela{f_Y}{c}(t)/\Mela{g}{c}(t)$ for all $t\in \Rz$ and rewrite \eqref{eq:lin:fun:1} into
\begin{align}\label{eq:lin:fun:2}
\vartheta(f) = \frac{1}{2\pi} \int_{-\infty}^{\infty} \Psi(-t) \frac{\Mela{f_Y}{c}(t)}{\Mela{g}{c}(t)} dt. 
\end{align}

A naive approach is to replace in  \eqref{eq:lin:fun:2} the
quantity  $\Mela{f_Y}{c}$ by its empirical counterpart
$\widehat{\mathcal M}_{c}(t):= n^{-1} \sum_{j=1}^n Y_j^{c-1+it}$,
$t\in \Rz$. However, the resulting integral is not well-defined, since
$\Psi\widehat{\mathcal M}_{c}/\Mela{g}{c}$ is generally not
integrable. We ensure integrability introducing an additional spectral cut-off regularisation which leads to the following estimator
\begin{align}
\widehat \vartheta_k:= \frac{1}{2\pi} \int_{-k}^k \Psi(-t) \frac{\widehat{\mathcal M}_{c}(t)}{\Mela{g}{c}(t)} dt \quad\text{ for any } k\in \pRz.
\end{align}

The following proposition shows that the estimator is consistent for suitable choice of the cut-off parameter $k\in \pRz$. We denote by $\E_{f_Y}^n$ the expectation corresponding to the distribution $\mathbb P_{f_Y}^n$ of $(Y_1, \dots, Y_n)$ and use the abbreviation $\E_{f_Y}:=\E_{f_Y}^1$. Analogously, we define $\E_{f}^n$ and $\E_f$.
\begin{pr}\label{pr:pw:upb}
	For $c\in \Rz$ assume that $f\in \LpA(x^{2c-1})$, $\Psi\Mela{f}{c}\in \Lz^1 _{\Rz}$ and $\E_{f_Y}(Y_1^{2(c-1)})<\infty$. Then for any $k \in \pRz$ holds
	\begin{align}\label{eq:bound:1}
	\hspace*{-1cm}\E_{f_Y}^n((\widehat \vartheta_k- \vartheta(f))^2) \leq \|\mathds 1_{[k, \infty)} \Psi\mathcal M_c[f]\|_{\mathbb L^1_{\mathbb R}}^2 + \frac{\E_{f_Y}(Y_1^{2(c-1)})}{ n} \|\1_{[-k,k]}\Psi/\mathcal M_c[g]\|_{\Lz^1_{\Rz}}^2
	\end{align}
	If additionally $\|g\|_{\infty,x^{2c-1}} <\infty$ holds, we get
	\begin{align}\label{eq:bound:2}
	\hspace*{-1cm}\E_{f_Y}^n((\widehat \vartheta_k- \vartheta(f))^2) \leq \|\mathds 1_{[k, \infty)} \Psi\mathcal M_c[f]\|_{\mathbb L^1_{\mathbb R}}^2 +  \frac{\|g\|_{\infty,x^{2c-1}} }{n} \E_{f}(X_1^{2(c-1)}) \Delta_{\Psi,g}(k) 
	\end{align}
	where $\Delta_{\Psi,g}(k):= \frac{1}{2\pi}\int_{-k}^k \left|\frac{\Psi(t)}{\Mela{g}{c}(t)}\right|^{2}dt$.
\end{pr}
 Choosing now a sequence of spectral cut-off parameters $(k_n)_{n\in \mathbb N}$ such that $k_n \rightarrow \infty$ and $\|\1_{[-k_n,k_n]}\Psi/\mathcal M_c[g]\|_{\Lz^1(\Rz)}^2 n^{-1} \rightarrow 0$ (respectively $\Delta_{\Psi,g}(k_n) n^{-1} \rightarrow 0$) implies that $\widehat{\vartheta}_{k_n}$ is a consistent estimator of $\vartheta(f)$, that is $\E_{f_Y}^n((\widehat \vartheta_{k_n}-\vartheta(f))^2) \rightarrow 0$ for $n\rightarrow \infty$. We note that the additional assumption, $\|g\|_{\infty, x^{2c-1}}< \infty$, is fulfilled by many error densities and thus rather weak. 
 \begin{rem}
 	Despite the fact, that the first bound \eqref{eq:bound:1} only
        requires a finite second moment of $Y_1^{c-1}$, we have in
        many cases $\Delta_{\Psi,g}(k)  \|\mathds 1_{[-k.k]}\Psi /
        \mathcal M_c[g]\|_{\mathbb L^1_{\mathbb R}}^{-2}\rightarrow 0$ for $k\rightarrow \infty$, implying that the bound of the variance term in \eqref{eq:bound:2} increases slower in $k$ than the bound presented in \eqref{eq:bound:1}. It is worth stressing out, that there exist cases where the opposite effect occurs. For instance let the error $U$ be lognormal-distribution with parameter $\mu=0, \lambda=1$, see  \cref{ex:mel:well}. Then
 	$\sup_{y\in \pRz} y^{2c-1} g(y) = \sup_{z\in \Rz} \frac{\exp(2(c-1)z)}{\sqrt{2\pi}\lambda} \exp(-(z-\mu)^2/2\lambda^2))<\infty $. Thus if $\E(X_1^{2(c-1)})<\infty$ both bounds are finite and following the argumentation of \cite{ButuceaTsybakov2008} one can see, that in the special case of  point-wise density estimation, the inequality presented in \eqref{eq:bound:1} is more favourable than the inequality presented in \eqref{eq:bound:2}. 
 \end{rem}
For the upcoming theory, we will focus on the second bound of \cref{pr:pw:upb}.  Assuming that $\| g\|_{\infty, x^{2c-1}} <\infty$, allows us to state that the growth of the second summand, also referred as variance term, is determined by the growth of $\Delta_{\Psi,g}(k)$ as $k$ going to infinity.

\paragraph{The parametric case}
In this paragraph we determine when \cref{pr:pw:upb} implies a
parametric rate of the estimator.  To be precise, there are two
scenarios only which occur.
\begin{resListeN}
	\item[\textcolor{blue}{(P)}] If $\sup_{k\in \pRz}
          \Delta_{\Psi, g}(k) =\|\Psi \mathcal M_c[g]^{-1}\|_{\mathbb
            L^2(\mathbb R)}^2< \infty$, i.e. the second summand in
          \eqref{eq:bound:2} is uniformly bounded in $k$ and hence of
           order $n^{-1}$. Then for all
          sufficiently large values of $k\in \mathbb R_+$  the bias
          term is negligible with respect to the parametric rate $n^{-1}$.
	\item[\textcolor{blue}{(NP)}] If $\sup_{k\in \pRz}
          \Delta_{\Psi, g}(k)  =\|\Psi \mathcal M_c[g]^{-1}\|_{\mathbb
            L^2(\mathbb R)}^2=\infty,$ i.e. the second summand is
          unbounded and hence necessitates an optimal choice of
          parameter $k\in \Rz^+$ realising to squared-bias-variance
          trade-off.
\end{resListeN}
Our aim is now to characterise when the case \textcolor{blue}{(P)} occur. To do so, we start by introducing a typical characterisation of the decay of the error density and the decay of the function $\Psi$, starting with the error density. 
 Let us first revisit Example \ref{ex:mel:well} to analyse the decay of the presented densities. 
\begin{ex}[Example \ref{ex:mel:well} continued]\label{ex:decay}
	\begin{exListeN}
		\item \textit{Beta Distribution}: For $c>0$ and $b\in \Nz$ we have $\mathcal M_c[g_b](t)= \prod_{j=1}^{b} \frac{j}{c-1+j+it}$ for  $t\in \Rz$ and thus
		\begin{align*}
		c_{g,c} (1+t^2)^{-b/2} \leq	|\mathcal M_c[g_b](t)| \leq C_{g,c} (1+t^2)^{-b/2} \quad t\in \Rz
		\end{align*}
		where $c_{g,c}, C_{g, c} >0$ are positive constants only depending on $g$ and $c$.
		\item \textit{Scaled Log-Gamma Distribution}: For $\lambda, a\in \pRz$, $\mu \in \Rz$ and $c<\lambda+1$ we have $\mathcal M_c[g_{\mu,a,\lambda}](t)=  \exp(\mu(c-1+it)) (\lambda-c+1-it)^{-a}$ for  $t\in \Rz.$
		\begin{align*}
	c_{g,c} (1+t^2)^{-a/2} \leq	|\mathcal M_c[g_{\mu,a,\lambda}](t)| \leq C_{g,c} (1+t^2)^{-a/2} \quad t\in \Rz
	\end{align*}
	where $c_{g,c}, C_{g, c} >0$ are positive constants only depending on $g$ and $c$.
		\item \textit{Gamma Distribution}: For $d\in \pRz$ and $c>-d+1$ we have $	\mathcal M_c[g_d](t)= \frac{\Gamma(c+d-1+it)}{\Gamma(d)}$ for  $t\in \Rz$ and thus 
			\begin{align*}
		 \hspace*{-1cm}c_{g,c} (1+t^2)^{(c+d-1.5)/2} \exp(-|t|\pi/2) \leq	|\mathcal M_c[g_d](t)| \leq C_{g,c} (1+t^2)^{(c+d-1.5)/2} \exp(-|t|\pi/2) 
		\end{align*}
		for $ t\in \Rz$ where $c_{g,c}, C_{g, c} >0$ are positive constants only depending on $g$ and $c$.
		\item \textit{Weibull Distribution:} Let $m\in \pRz$ and $c>-m+1$ we have $	\mathcal M_c[g_m](t)= \frac{(c-1+it)}{m}\Gamma\left(\frac{c-1+it}{m}\right)$ for $t\in \Rz$ and thus 
			\begin{align*}
		\hspace*{-1cm}c_{g,c} (1+t^2)^{\frac{2c-2-m}{2m}} \exp(-\frac{|t|\pi}{2m}) \leq	|\mathcal M_c[g_m](t)| \leq C_{g,c}(1+t^2)^{\frac{2c-2-m}{2m}}  \exp(-\frac{|t|\pi}{2m})
		\end{align*}
			for $ t\in \Rz$ where $c_{g,c}, C_{g, c} >0$ are positive constants only depending on $g$ and $c$.
		\item \textit{Lognormal Distribution:} Let $\lambda \in \pRz$, $\mu \in \Rz$ and $c\in \Rz$ we have $\Mela{g_{\mu,\lambda}}{c}(t)= \exp(\mu(c-1+it))\exp\left(\frac{\lambda^2(c-1+it)^2}{2}\right)$ for $t\in \Rz$ and thus
	\begin{align*}
	\hspace*{-1cm}c_{g,c}\exp(-\lambda^2 t^2/2)\leq	|\mathcal M_c[g_m](t)| \leq C_{g,c}\exp(-\lambda^2 t^2/2)
		\end{align*}
			for $ t\in \Rz$ where $c_{g,c}, C_{g, c} >0$ are positive constants only depending on $g$ and $c$.
	\end{exListeN}
	\end{ex}

Motivated by \cref{ex:mel:well} we distinguish between \textit{smooth error } and \textit{supersmooth error densities}  staying in the terminology of \cite{Fan1991}, \cite{BelomestnyGoldenshlugerothers2020} or \cite{Brenner-MiguelComteJohannes2020}.
An error density $g$ is called \textit{smooth} if there exists a $\gamma, c_{g,c}, C_{g, c}>0$ such that
\begin{align*}
c_g(1+t^2)^{-\gamma/2} \leq |\Mela{g}{c}(t)| \leq C_{g,c} (1+t^2)^{-\gamma/2}, \quad t\in \Rz \tag{\textbf{[G1]}}
\end{align*}
and it is referred to as \textit{super smooth} if there exists $ \lambda, \rho, c_{g,c}, C_{g, c}>0$ and $\gamma\in \Rz$ such that
\begin{align*}
\hspace*{-1cm} c_{g,c}(1+t^2)^{-\gamma/2} \exp(-\lambda |t|^{\rho}) \leq |\Mela{g}{c}(t)| \leq C_{g,c}(1+t^2)^{-\gamma/2} \exp(-\lambda |t|^{\rho}) ,\  t\in \Rz. \tag{\textbf{[G2]}}
\end{align*}
On the other hand to calculate the growth of $\Delta_{\Psi,g}$ we
specify the decay of $\Psi$. Similar to the error density $g$ we
consider the case of a \textit{smooth} $\Psi$, i.e. there exists $c_{\Psi,c}, C_{\Psi,c}>0$ and $p\geq 0$ such that
\begin{align*}
c_{\Psi,c} (1+t^2)^{-p/2} \leq |\Psi(t)| \leq C_{\Psi, c} (1+t^2)^{-p/2}, \quad t\in \Rz,  \tag{$[\bm{\Psi 1]}$}
\end{align*}
and a \textit{super smooth} $\Psi$, i.e. there exists $\mu, R, c_{\Psi,c}, C_{\Psi, c}>0$ and $p\in \Rz$ such that
\begin{align*}
\hspace*{-1cm} c_{\Psi, c}(1+t^2)^{-p/2} \exp(-\mu |t|^{R}) \leq |\Mela{\Psi}{c}(t)| \leq C_{\Psi,c}(1+t^2)^{-p/2} \exp(-\mu |t|^{R})  ,\ t\in \Rz. \tag{$[\bm{\Psi 2]}$}
\end{align*}
As we see in  the following Illustration the examples of $\Psi$ considered in \cref{il:lin:fun:example} do fit into these two cases.
\begin{il}[\cref{il:lin:fun:example} continued]\label{il:gamma:psi}
	\begin{exListeN}
		\item\textit{Point-wise density estimation:} We have that $|\Psi(t)|=x_o^{-c}$ and thus $p=0$ in sense of $[\bm{\Psi 1]}$. 
		\item \textit{Point-wise cumulative distribution function estimation:} We have that $|\Psi(t)|=\frac{x_o^{1-c}}{\sqrt{(1-c)^2+t^2} }$ and thus $p=1$ in sense of $[\bm{\Psi 1]}$. 
		\item \textit{Point-wise survival function estimation:} We have that $|\Psi(t)|=\frac{x_o^{1-c}}{\sqrt{(1-c)^2+t^2} }$ and thus $p=1$ in sense of $[\bm{\Psi 1]}$.
		\item \textit{Laplace transform estimation:} We have that $|\Psi(t)|=t_o^{c-1}|\Gamma(1-c+it)|$ and thus $p=1-2c$, $\mu = \pi/2$ and $R=1$ in the sense of $[\bm{\Psi 2]}$.
	\end{exListeN}
\end{il}
After the introduction of the typical terminology for deconvolution settings we can state when the function $\Delta_{\Psi,g}$ is bounded. We summarize the collection of scenarios in the following Proposition.
\begin{pr}\label{pro:param:case}
	Assume that for a $c\in \Rz$ holds $f\in \LpA(x^{2c-1})$, $\Psi\Mela{f}{c}\in \Lz^1 _{\Rz}$,  $\sigma=\E_{f}(X_1^{2(c-1)})<\infty$ and $\|g\|_{\infty, x^{2c-1}}< \infty$. Then for the cases
	\begin{resListeN}
		\item $[\bm{\Psi 1]}$ and \textbf{[G1]} with $2p -2\gamma >1$;
		\item $[\bm{\Psi 2]}$ and \textbf{[G1]} or
		\item $[\bm{\Psi 2]}$ and \textbf{[G2]} with $(R>\rho)$, $(R=\rho, \mu>\lambda)$ or $(R=\rho, \mu = \lambda, 2p-2\gamma >1)$
		\end{resListeN}
	we get that $\sup_{k\in \pRz}
        \Delta_{\Psi,g}<\infty$. Furthermore, for all $k\in \mathbb R$
        sufficiently large we have
	\begin{align*}
	 		\E_{f_Y}^n((\widehat \vartheta_{k}- \vartheta(f))^2) \leq \frac{C(\Psi,g,\sigma)}{n}.
	\end{align*}
	\end{pr}
The proof of \cref{pro:param:case} is a straight forward calculus and thus omitted.
For our four examples of $\Psi$ we get a parametric rate for the estimation of the survival function and cumulative distribution function if the error density fulfils \textbf{[G1]} with $\gamma<1/2$ and a parametric rate for the estimation of the Laplace transform if the error density fulfils  \textbf{[G1]} with  $\gamma>0$ or if $g$ fulfils \textbf{[G2]} with ($\rho <1$), ($\rho = 1$, $\lambda < \pi/2), $($\rho = 1$ or $\lambda =\pi/2, \gamma< -c)$. 
\paragraph{The non-parametric case}
We  now focus on the case \textcolor{blue}{(NP)},
that is $\sup_{k\in \pRz} \Delta_{\Psi,g}(k) = \infty$, which occurs
in several situations. In this scenario the first summand of
\cref{pr:pw:upb} is decreasing in $k$ while the second summand is
increasing and unbounded. A choice of the parameter $k\in
\pRz$ realising an optimal trade-off  is thus non-trivial. We therefore define a data-driven procedure for the choice of the parameter $k\in \pRz$ inspired by the work of \cite{GoldenshlugerLepski2011}.\\
In fact, let us reduce the set of possible parameters to $\mathcal
K_n:=\{k\in \Nz: \|g\|_{\infty,x^{2c-1}}\Delta_{\Psi,g}(k)\leq n, k
\leq n^{1/2}(\log n)^{-2}\}$ and denote $ K_n = \max \mathcal K_n$. We
further introduce the variance term up to a $(\log n)$-term\begin{align*}
V(k):= \chi\|g\|_{\infty, x^{2c-1}} \sigma \Delta_{\Psi,g}(k)(\log n)n^{-1}
\end{align*}
where $\chi>0$ is a numerical constant which is specified below and
$\sigma:= \E_{f}(X_1^{2(c-1)})$. Based on a comparison of the
estimators constructed above an estimator of the bias term is given by
\begin{align*}
A(k):= \sup_{k'\in \nsetlo{k, K_n}} ((\widehat \vartheta_{k'}-\widehat \vartheta_k)^2-V(k'))_+
\end{align*}
where $\nsetlo{a,b}:=(a,b]\cap \mathbb N$ for $a,b\in \pRz.$ Analogously, we define $\nset{a,b}=[a,b]\cap \mathbb N$ and $\nsetro{a,b}=[a,b)\cap \mathbb N$.
Since the term $\sigma$ in $V(k$) depends on the unknown density $f$,
and hence it
is itself unknown, we replace it by  the plug-in estimator
$\widehat{\sigma}:= \frac{1}{n} \sum_{j=1}^n
\frac{Y_j^{2(c-1)}}{\E(U_1^{2(c-1)})}$. Summarising we estimate $V(k)$
and $A(k)$ by
\begin{align*}
\widehat{V}(k):=2 \chi \|g\|_{\infty, x^{2c-1}} \widehat\sigma \Delta_{\Psi,g}(k) \log(n)n^{-1} \text{ and } \widehat A(k):= \sup_{k'\in \nsetlo{k, K_n}} ((\widehat \vartheta_{k'}-\widehat \vartheta_k)^2-\widehat V(k'))_+.
\end{align*}
Below we study the fully data-driven estimator $\widehat
\vartheta_{\widehat k}$ of $\vartheta(f)$ with % $\widehat k$ is given by 
\begin{align*}
\widehat k:= \argmin_{k\in \mathcal K_n} (\widehat A(k)+ \widehat V(k)).
\end{align*}
\begin{thm}\label{theo:adap:functional}
	 For $c\in \Rz$ assume that $f\in \LpA(x^{2c-1})$, $\Psi\Mela{f}{c}\in \Lz^1 _{\Rz}$, $\E_{f_Y}(Y_1^{8(c-1)})<\infty$ and $\|g\|_{\infty,x^{2c-1}}<\infty$. Then for $\chi \geq 72 $  holds
	\begin{align*}
	\E_{f_Y}((\vartheta(f)- \widehat{\vartheta}_{\widehat k})^2) \leq C_1 \inf_{k\in \mathcal K_n} ( \|\mathds 1_{[k,\infty)} \Psi \mathcal M_c[f]\|_{\Lz^1_{\Rz}}^2 +V(k) )+ \frac{C_2}{n}
	\end{align*}
	where $C_1$ is a positive numerical constant and $C_2$ is a positive constant depending on  $\Psi$, $g$, $\E_{f_Y}(Y^{8(c-1)})$.
\end{thm}

The proof of \cref{theo:adap:functional} is postponed to the appendix. Let us shortly comment on the moment assumptions of \cref{theo:adap:functional}. For $c\in \mathbb R$ close to one, the apparently high moment assumption $\E_{f_Y}(Y_1^{8(c-1)})<\infty$ is rather weak. For the point-wise density estimation, compare \cref{il:lin:fun:example},  this assumption is always true if $c=1$. For the point-wise survival function estimation (respectively. cumulative distribution function estimation), $c=1$ cannot be fullfilled but arbitrary values of $c\in \mathbb R$ close to one are possible. 
As already mentioned, for the pointwise density estimation the assumption $\Psi\mathcal M_1[f] \in \Lz^1_{\Rz}$ implies that $\mathcal M_1[f] \in \Lz^1_{\Rz}$.  For $c=1$, we see that $\|g\|_{\infty,x}< \infty$ is fullfilled for many examples of error densities. 

%% file: _3Minimax.tex
\section{Minimax theory}\label{mm}
In the following section we develop the minimax theory for the plug-in
spectral cut-off estimator under the assumptions \textbf{[G1]} and
[$\bm{\Psi1}$]. Over the \textit{Mellin-Sobolev} spaces we derive an
upper for all linear functional satisfying assumption
[$\bm{\Psi1}$]. We state a lower bound for each of the cases
\ref{il:lin:fun:example:i}-\ref{il:lin:fun:example:iii} of
\cref{il:lin:fun:example} separately, that is point-wise estimation of the density $f$, the survival function $S$ and the cumulative distribution function $F$.  We finish this section, by motivating the regularity spaces through their analytically implications.

\paragraph{Upper bound}
Let us restrict to the scenario where [\textbf{G1}] and $\bm{[\Psi1]}$ holds for $2p-2\gamma \leq 1$.
Here one can state that there exist a constant $C_{\Psi,g}>0$ such that $\Delta_{\Psi,g}(k) \leq C_{\Psi,g} k^{2\gamma-2p+1}$. Now let us consider the bias term. To do so, we introduce we \textit{Mellin-Sobolev spaces} at the development point $c\in \Rz$ by
\begin{align}\label{eq:Mel:Sob}
\Wz^s_c(\pRz):=\{ h \in \LpA(x^{2c-1}): |h|_{s,c}^2:= \|(1+t^2)^{s/2} \Mela{h}{c}\|_{\Lz^2_{\Rz}}< \infty\}
\end{align}
with corresponding ellipsoids $\Wz^s_c(L):=\{h \in \Wz^s_c(\pRz): |h|_{s,c}^2 < L\}$.
We denote the subset of densities by 
\begin{equation}\label{equation:dens_sobol}
\rwcSobD{\wSob,c,\rSob}:= \{f\in \Wz_c^s(L): f \text{ is
	a density}, \E_{f}(X_1^{2c-2}) \leq L\}.
\end{equation}
Using this construction we get the following result as a direct consequence.
\begin{thm}\label{thm:pw:rt}
	Assume that for a $c\in \Rz$ [\textbf{G1}] holds for $g$ and [$\bm{\mathit{\Psi1}}$] for $\Psi.$ Additionally, assume that  $\| g\|_{\infty,x^{2c-1}}<\infty$. Setting for any $s>1/2-p$ the cut-off parameter to $k_n:= n^{1/(2s+2\gamma)}$ implies then
	\begin{align*}
	\sup_{f \in \rwcSobD{\wSob,c,\rSob}} \E_{f_Y}((\widehat{\vartheta}_{k_n}-\vartheta(f))^2) \leq C_{L ,s,\Psi,g, c}\, n^{-(2s+2p-1)/(2s+2\gamma)}
	\end{align*}
	where $C_{L ,s,\Psi,g,c}>0$ is a constant depending on $L,s,p,\Psi,\gamma, c$ and $\| g\|_{\infty,x^{2c-1}}$.
\end{thm}
\begin{proof}[Proof of \cref{thm:pw:rt}]
	Evaluating  the upper bound in \cref{pr:pw:upb} 
	under \textbf{[G1]} and [$\bm{\mathit{\Psi1}}$] we have
	$
	\|g\|_{\infty,x^{2c-1}}\sigma \Delta_{\psi,g}(k)n^{-1} \leq C_{\Psi, g, L} \frac{k^{2\gamma-2p+1}}{n} 
	$
	and 
	\begin{align*}
	\hspace*{-0.7cm}\left(\int_k^{\infty} |\Psi(t)\Mela{f}{c}(t)| dt \right)^2
	 \leq C_{L,c} \int_k^{\infty} |\Psi(t)|^2 (c^2+t^2)^{-s} dt \leq C_{L, c, s, \psi} k^{-2s-2p+1}.
	\end{align*}
	Now choosing $k_n:= n^{1/(2s+2\gamma)}$ balances both term leading to the rate $n^{-(2s+2p-1)/(2s+2\gamma)}$.
\end{proof}
The assumption $s>1/2-p$ implies that $\Psi \Mela{f}{c} \in \Lz^1_{\Rz}$ by a simple calculus which can be found in proof of \cref{thm:pw:rt} in the appendix. Before considering the lower bounds let us illustrate the last Theorem using our examples (i) to (iii) of \cref{il:lin:fun:example}. 
\begin{il}\label{il:mini:max}
	\begin{exListeN}
		\item\label{il:mini:max:i}\textit{Point-wise density estimation:} Since $p=0$ we  assume that $s>1/2=1/2-p$. In this scenario \cref{thm:pw:rt} implies
		\begin{align*}
		\sup_{f \in \rwcSobD{\wSob,c,\rSob}} \E_{f_Y}((\widehat{\vartheta}_{k_n}-\vartheta(f))^2) \leq x_o^{-2c} C_{L ,s,\Psi,g, c}\, n^{-(2s-1)/(2s+2\gamma)}.
		\end{align*} 
		\item\label{il:mini:max:ii}\textit{Point-wise cumulative distribution 
                    function estimation:} We have $p=1$ and hence for
                  any $s\geq 0$ holds $s>1/2-p$. Recall that for $\gamma <1/2$
                  we are in the parametric case  where we  choose $k\in \mathbb R_+$ sufficiently large. For $\gamma \geq 1/2$ we deduce from \cref{thm:pw:rt} for any $c<1$ that
		\begin{align*}
		\sup_{f \in \rwcSobD{\wSob,c,\rSob}} \E_{f_Y}((\widehat{\vartheta}_{k_n}-\vartheta(f))^2) \leq x_o^{2-2c} C_{L ,s,\Psi,g, c}\, n^{-(2s-1)/(2s+2\gamma)}.
		\end{align*} 
		\item\label{il:mini:max:iii}\textit{Point-wise survival function estimation:} 
                  We have $p=1$ and hence for any $s\geq 0$ holds
                  $s>1/2-p$. Recall that for $\gamma <1/2$
                  we are in the parametric case  where we  choose $k\in \mathbb R_+$ sufficiently large. For $\gamma \geq 1/2$ we deduce from \cref{thm:pw:rt} for any $c<1$ that
		\begin{align*}
		\sup_{f \in \rwcSobD{\wSob,c,\rSob}} \E_{f_Y}((\widehat{\vartheta}_{k_n}-\vartheta(f))^2) \leq x_o^{2-2c} C_{L ,s,\Psi,g, c}\, n^{-(2s-1)/(2s+2\gamma)}.
		\end{align*}
              \end{exListeN}
              In example \ref{il:mini:max:i} the sign of $c$ has a
              strong impact on the upper bound. In fact, for $c>0$ it
              appears that the estimation in a point $x_o$ close to
              $0$ is harder than for bigger values of $x_o$. The case for $c<0$
              has an opposite effect. Further in \ref{il:mini:max:ii}
              and \ref{il:mini:max:iii}, i.e. estimating the survival
              function and the c.d.f. estimation, the  estimator of the
              c.d.f. seems to have a better behaviour close to 0 than
              the survival function estimator. We stress out, that in
              \cref{il:lin:fun:example} we already mention that one
              can use an estimator for the survival function to
              construct an estimator for the c.d.f and vice versa. The
              results of \cref{il:mini:max} suggests to estimate the
              survival function directly or using the
              c.d.f. estimator, according if $x_o\in \pRz$ is close to
              0 or not.
\end{il}

\begin{rem}\label{rem:bg:comp}
  \cite{BelomestnyGoldenshlugerothers2020} derive for point-wise density estimation a rate of
  $n^{-2s/(2s+2\gamma+1)}$ under similar assumptions on the error
  density $g$. However, they consider  Hölder-type regularity classes
  rather than Mellin-Sobolev spaces which are of a global nature. Even
  if the rates in \cref{il:mini:max}  seem to be less sharp
  compared to \cite{BelomestnyGoldenshlugerothers2020}, they cannot be  improved  as shown by the lower bounds below.
\end{rem}

Additionally, if $\gamma \geq 1$ we have that $k_n:=n^{1/(2s+2\gamma)} \leq k^{1/2}$ and thus $k_n \in \mathcal K_n$. We can deduce the following Corollary using the similar arguments of the proof of \cref{thm:pw:rt} on \cref{theo:adap:functional}. We therefore omit its proof.

\begin{co}\label{co:adap:sobo}
	Assume that for a $c\in \Rz$ holds  [\textbf{G1}] holds for $g$ and [$\bm{\mathit{\Psi1}}$] for $\Psi.$ Further let $\E_{f_Y}(Y_1^{8(c-1)}),$ $\|g\|_{\infty,x^{2c-1}}<\infty$ and $f \in \rwcSobD{\wSob,c,\rSob}$ for any $s>1/2-p$. Then
	\begin{align*}
	\E_{f_Y}((\widehat{\vartheta}_{\widehat k}-\vartheta(f))^2) \leq C_{f,g,\Psi}\log(n)n^{-(2s+2p-1)/(2s+2\gamma)}
	\end{align*}
	where $C_{f,g,\Psi}>0$ is a  constant depending on $L,$$s,$$p,$$\Psi,$$\gamma$,$ c$, $ \E_{f_Y}(Y_1^{8(c-1)})$ and $\|g\|_{\infty,x^{2c-1}}$.
\end{co}

To state that the presented rates of  \cref{thm:pw:rt} cannot be improved over the whole Mellin-Sobolev ellipsoids, we give a lower bound result for the cases (i)-(iii) in the following section.
\paragraph{Lower bound}
For the following part, we will need to have an additionally assumption on the error density $g$. In fact, we will assume that $g$ has bounded support, that is $g(x)=0$ for $x>d$, $d\in\pRz$. For the sake of simplicity we will say that $d=1$.  Further we assume that there exists $c_g', C_g'\in \pRz$ such that
 \begin{align*}
 \tag{[\textbf{G1'}]}    c_g' (1+t^2)^{-\gamma/2} \leq | \Melop_{1/2}[g](t)| \leq C_g' (1+t^2)^{-\gamma/2} \text{ for } |t|\rightarrow \infty.
\end{align*}
For technical reasons we will restrict ourselves to the case of $c>1/2$.

\begin{thm}\label{theorem:lower_bound_pw}
	Let $\wSob, \gamma \in\Nz$ , assume that \textbf{[G1]} and \textbf{[G1']} holds . Then there exist
	constants $\cst{g,x_o,i},L_{\wSob,g,x_o, c,i}>0,$$ i\in \nset{3},$ such that 
	\begin{resListeN}
		\item \textit{Point-wise density estimation:} for all
		$L\geq L_{\wSob,g,x_o, c,1}$, $n\in \Nz$ and for any estimator $\widehat f(x_o)$ of $f(x_o)$ based
		on an i.i.d. sample $\Nsample{Y_j}$, 
		\begin{align*}
		\sup_{f\in\rwcSobD{\wSob,c,\rSob}}\E_{f_Y}^n((\widehat f(x_o)-f(x_o)^2) \geq \cst{g,x_o,1} n^{-(2s-1)/(2s+2\gamma)}.
		\end{align*}
		\item\textit{Point-wise survival function estimation:}  for all
		$L\geq L_{\wSob,g,x_o, c,2}$, $n\in \Nz$ and for any estimator $\widehat S(x_o)$ of $S(x_o)$ based
		on an i.i.d. sample $\Nsample{Y_j}$, 
		\begin{align*}
		\sup_{f\in\rwcSobD{\wSob,c,\rSob}}\E_{f_Y}^n((\widehat S(x_o)-S(x_o)^2) \geq \cst{g,x_o,2} n^{-(2s+1)/(2s+2\gamma)}.
		\end{align*}
		\item \textit{Point-wise cumulative distribution function estimation:} for all
		$L\geq L_{\wSob,g,x_o, c,3}$, $n\in \Nz$ and for any estimator $\widehat F(x_o)$ of $F(x_o)$ based
		on an i.i.d. sample $\Nsample{Y_j}$, 
		\begin{align*}
		\sup_{f\in\rwcSobD{\wSob,c,\rSob}}\E_{f_Y}^n((\widehat F(x_o)-F(x_o)^2) \geq \cst{g,x_o,3} n^{-(2s+1)/(2s+2\gamma)}.
		\end{align*}
	\end{resListeN}
\end{thm}

We want to stress out that in the multiplicative censoring model, the family $(g_k)_{k\in \Nz}$ of $\text{Beta}_{(1,k)}$ densities fulfils both assumption \textbf{[G1]} and \textbf{[G1']}.

\paragraph{Regularity assumptions} While in the theory of inverse
problems the definition of the Mellin-Sobolev spaces is quite
natural, we want to stress out that elements of these spaces can be characterised by their analytical properties. In \cite{Brenner-MiguelComteJohannes2020} one can find a characterisation of $\mathbb W_1^s(\pRz)$. Since the generalisation for the spaces $\Wz_c^s(\pRz)$ is straight forward, we only state the result while the proof for the case $c=1$ can be found in \cite{Brenner-MiguelComteJohannes2020}.

\begin{pr}\label{mm:pr:reg}
	Let $s \in \Nz$. Then $f \in \Wz_c^s(\pRz)$ if and only if $f$ is $s-1$-times continuously differentiable where $f^{(s-1)}$ is locally absolutely continuous with derivative $f^{(s)}$ and  $\omega^j f^{(j)} \in \LpA(\omega^{2c-1})$ for all $j\in \nset{0,s}$.
\end{pr}

%Now we want to briefly state the adaptivity result. In fact using the construction $\mathcal K_n, \widehat k, \gamma_n$ and $(\mathrm{pen}(k))_{k\in \mathcal K_n}$ introduced in \cref{ag} we can state the following result as an immediate implication of \cref{dd:thm:ada} and \cref{mm:thm:upper}.
%\begin{thm}\label{mm:thm:ad}
	%Assume that $f\in \LpA \cap \LpA(\omega)$, \textbf{[G1]} and that $\|\omega f_Y\|_{\infty} := \sup_{y>0} |y f_Y| <\infty$. Then for $\chi > 24 C_g \pi^{-1}$ 
	%\begin{align*}
	%\E_{f_Y}^n (	\| f- \widehat f_{\widehat k}^n \|^2) \leq 4\inf_{k\in\mathcal K_n}\big(\|f-f_k\|_{\omega}^2 +\mathrm{pen}(k) \big) + C(\|f\|_{\omega},\|\omega f_Y\|_{\infty},g)  n^{-1} 
	%\end{align*}
	%where $C(\|\omega f_Y\|_{\infty},g)>0$ is a constant depending on $\|\omega f_Y\|_{\infty}$ and $g$.
%\end{thm}

\newpage
%%% Local Variables:
%%% mode: latex
%%% TeX-master: "_0DANDER"
%%% End:

%% file: _app2.tex
%======================================================================================================================
%                                                                 
% Title: Appendix 3
% Author: Sergio Brenner Miguel. Fabienne Comte,Jan JOHANNES
% 
% Date: %%ts latex start%%[2020-07-30 Thu 19:49]%%ts latex end%%
%
% ======================================================================================================================
\subsection{Proofs of \cref{dd}}\label{a:dd}
% --------------------------------------------------------------------
% <<Proof of Re key argument>>
% --------------------------------------------------------------------

\paragraph{Usefull inequality}\label{a:in}
The next inequality was is state in the following form in \cite{Comte2017} based on a similar formulation in \cite{BirgeMassart1998}.
\begin{lem}(Bernstein inequality)\label{bsi:re}
	Let $X_1, \dots, X_n$ independent random variables and $T_n(X):= \sum_{j=1}  (X_i- \E(X_i))$. Then for $\eta>0$,
	\begin{align*}
	\Pz(|T_n(X)-\E(T_n(X))| \geq n \eta) \leq 2 \exp(-\frac{n \frac{\eta^2}{2}}{v^2 + b\eta}) \leq 2 \max(\exp(-\frac{n\eta^2}{4v^2}, \exp(- \frac{n\eta}{4b})))
	\end{align*}
	if $n^{-1} \sum_{i=1}^n \E(|X_i^m|) \leq \frac{m!}{2} v^2 b^{m-2}$ for all $m\geq 2$. If the $X_i$ are identically distributed, the previuos condition can be replaced by $\Var(X_1) \leq v^2$ and $|X_1| \leq b$.
\end{lem}

\begin{proof}[Proof of \cref{pr:pw:upb}]
	Let us denote for any $k\in \pRz$ the expectation $\vartheta_{k}:= \E_{f_Y}^n( \widehat{\vartheta}_k)$ which leads to the usual squared bias-variance decomposition 
	\begin{align}\label{eq:b:v:dec}
	\E_{f_Y}^n((\widehat{\vartheta}_k-\vartheta(f))^2) = (\vartheta_k-\vartheta(f))^2+ \Var_{f_Y}^n(\widehat{\vartheta}_k).
	\end{align} 
	Consider  the first summand in \eqref{eq:b:v:dec}- An application of the Fubini-Tonelli theorem implies
	\begin{align*}
	(\vartheta_k-\vartheta(f))^2=\left( \frac{1}{2\pi} \int_{[-k,k]^c} \Psi(-t) \Mela{f}{c}(t) dt\right)^2 \leq \|\mathds 1_{[k, \infty)} \Psi\mathcal M_c[f]\|_{\Lz^1_{\Rz}}^2
	\end{align*}
	Study the the second term in \eqref{eq:b:v:dec}. Then the bound in \eqref{eq:bound:1} follows then by the following inequality
	\begin{align*}
	\hspace*{-1cm}\Var_{f_Y}^n (\widehat{\vartheta}_k)
	&\leq \frac{1}{n} \E_{f_Y}((\frac{1}{2\pi} \int_{-k}^k |\Psi(t)| \frac{Y_1^{c-1}}{|\Mela{g}{c}(t)|}dt)^2 )
	&= \frac{\E_{f_Y}(Y_1^{2(c-1)})}{4\pi^2n}\left( \int_{-k}^k \left|\frac{\Psi(t)}{\Mela g c(t)}\right| dt \right)^2.
	\end{align*}
	To show \eqref{eq:bound:2} we see that
	\begin{align*}
	\Var_{f_Y}^n (\widehat{\vartheta}_k)& \leq \frac{1}{n} \E_{f_Y}((\frac{1}{2\pi} \int_{-k}^k \Psi(-t) \frac{Y_1^{c-1+it}}{\Mela{g}{c}(t)}dt)^2 ) \\
	&=\frac{1}{n} \int_{0}^{\infty} f_Y(y) \left|\frac{1}{2\pi} \int_{-k}^k \Psi(-t) \frac{y^{c-1+it}}{\Mela{g}{c}(t)}dt \right|^2 dy\\
	&\leq \frac{\|f_Y\|_{\infty,x^{2c-1}}}{2\pi n}\int_{-k}^k \left|\frac{\Psi(t)}{\Mela g c(t)}\right|^2 dt.
	\end{align*}
	 Furthermore we have for any $y>0$ that
	\begin{align*}
	y^{2c-1} f_Y(y) = \int_0^{\infty} f(x)x^{2c-2} g(y/x) \frac{y^{2c-1}}{x^{2c-1}} dx \leq \| g\|_{\infty, x^{2c-1}} \E_{f}(X_1^{2c-2}). 
	\end{align*}
\end{proof}

\begin{proof}[Proof of \cref{theo:adap:functional}]
	 Let us set $\vartheta:=\vartheta(f)$. By the definition of $\widehat{k}$ follows for any $k\in \mathcal K_n$
	\begin{align*}
	(\vartheta-\widehat{\vartheta}_{\widehat k})^2 &\leq 2(\vartheta- \widehat \vartheta_k)^2 +2 (\widehat \vartheta_k- \widehat \vartheta_{\widehat k})^2 \\
	&\leq 2(\vartheta- \widehat \vartheta_k)^2 +2 (\widehat \vartheta_k- \widehat \vartheta_{k \wedge \widehat k} )^2+ 2(\widehat \vartheta_{k \wedge \widehat k}- \widehat \vartheta_{\widehat k})^2 \\
	&\leq 2(\vartheta- \widehat \vartheta_k)^2 + 2( \widehat A(\widehat k)+ \widehat V(k)+ \widehat A(k)+ \widehat V(\widehat k)) \\
	& \leq 2(\vartheta- \widehat \vartheta_k)^2 + 4( \widehat A(k)+ \widehat V(k)) . 
	\end{align*}
	Consider $\widehat A(k)$ we have by a straight forward calculus $\widehat A(k) \leq A(k)+ \sup_{k'\in \mathcal K_n} (V(k')-\widehat V(k'))_+$ and thus
	\begin{align*}
	 A(k) \leq \sup_{k'\in \nsetlo{k, K_n}} \left( 3(\widehat \vartheta_{k‘}-\vartheta_{k'})^2+3(\widehat \vartheta_{k}-\vartheta_{k})^2-V(k') \right)_+ +3\max_{k'\in \nsetlo{k, K_n}} (\vartheta_k-\vartheta_{k'})^2.
	\end{align*}
	By the monoticity of $ V(k)$ we deduce that for $k<k'$ holds $ V(k)\geq \frac{1}{2} V(k)+\frac{1}{2}V(k')$ which simplifies the term to
	\begin{align*}
	A(k) \leq 6\sup_{k'\in \nset{k, K_n}} \left( (\widehat \vartheta_{k‘}-\vartheta_{k'})^2- \frac{1}{6} V(k') \right)_+ +3\max_{k'\in \nsetlo{k, K_n}} (\vartheta_k-\vartheta_{k'})^2
	\end{align*}
	while the lattern summand can be bounded for any $k'\in \nsetro{k, K_n}$ by
	\begin{align*}
	(\vartheta_k-\vartheta_{k'})^2 \leq  \left(\frac{1}{2\pi}\int_{[-k',k']\setminus [-k,k]}\Psi(-t) \Mela{f}{c}](t) dt \right)^2 \leq  \frac{1}{\pi^2}(\int_{k}^{\infty}|\Psi(t) \Mela{f}{c}(t)|dt )^2.
	\end{align*}
	Further we have that $(\vartheta- \widehat\vartheta_k)^2 \leq 2(\vartheta- \vartheta_k)^2 +2 (\widehat \vartheta_k- \vartheta_k)^2$ which implies
	\begin{align*}
	\hspace*{-0.5cm}(\vartheta-\widehat{\vartheta}_{\widehat k})^2  &\leq \frac{16}{\pi^2} \left(( \int_k^{\infty} |\Psi(t)\Mela{f}{c}(t )|dt)^2 + V(k) \right)
	+ 4\sup_{k'\in \mathcal K_n} (V(k')-\widehat V(k'))_+\\
		 &+ 26\sup_{k'\in \nset{k, K_n}} \left( (\widehat \vartheta_{k'}-\vartheta_{k'})^2- \frac{1}{6}V(k')
	 \right)_+. \\
	\end{align*}
	To control the last term we split the centred arithmetic mean $\widehat \vartheta_k-\vartheta_k$ into two terms, applying at one term a Bernstein inequality,  cf lemma \ref{bsi:re}, and standard techniques on the other term. For a positive sequence $(c_n)_{n\in \Nz}$ and $t\in \Rz$ introduce
	\begin{align*}
	\hspace*{-0.5cm}\widehat{\mathcal M}_{c}(t) = n^{-1} \sum_{j=1}^n( Y_j^{c-1+it} \1_{(0,c_n)}(Y_j^{c-1})+ Y_j^{c-1+it}\1_{(c_n, \infty)}(Y_j^{c-1}))=: \widehat{\mathcal M}_{c,1}(t) + \widehat{\mathcal M}_{c,2}(t).
	\end{align*} 
	Split the centred arithmetic mean $\widehat{\vartheta}_k- \vartheta_k= \nu_{k,1}+\nu_{k,2}$ where $\nu_{k,i}:= \frac{1}{2\pi} \int_{-k}^k \frac{\Psi(-t)}{\Mela{g}{c}(t)} ( \widehat{\mathcal M}_{c,i}(t)- \E_{f_Y}(\widehat{\mathcal M}_{c,i}(t)))dt$. Thus we have 
	\begin{align*}
	(\vartheta-\widehat{\vartheta}_{\widehat k})^2  \leq \frac{16}{\pi^2} &\left((\int_k^{\infty} |\Psi(t)\Mela{f}{c}(t )|dt)^2+V(k) \right) + 4 \sup_{k'\in \mathcal K_n} (\widehat V(k')-V(k'))_+\\
	& + 52\sup_{k'\in \mathcal K_n} \left( \nu_{k',1}^2- \frac{1}{12}V(k') \right)_+ + 52\sup_{k'\in \mathcal K_n} \nu_{k',2}^2.
	\end{align*}
	The claim of the theorem follows thus by the following lemma.
	\begin{lem}\label{lem:concen}
		Under the assumptions of Theorem \ref{theo:adap:functional} with $c_n:=\sqrt{n^{1/2}\sigma \|g\|_{\infty, x^{2c-1}} \log(n)}/42$ hold
		\begin{align*}
			(i) \quad &\E_{f_Y} (\sup_{k'\in \mathcal K_n} \left( \nu_{k',1}^2- \frac{1}{12}V(k') \right)_+ ) \leq  \frac{C(\sigma, \|g\|_{\infty, x^{2c-1}})}{n},\\
			(ii) \quad &\E_{f_Y}(\sup_{k'\in \mathcal K_n} \nu_{k',2}^2 ) \leq \frac{C(\Psi,g,\sigma)}{n}  \text{ and }\\
			(iii)\quad &\E_{f_Y}(\sup_{k'\in \mathcal K_n}(V(k')-\widehat V(k'))_+) \leq\frac{C(\E(X_1^{4(c-1)} , \sigma)}{n}.
		\end{align*}
	\end{lem}
\end{proof}{
	
	\begin{proof}[Proof of \cref{lem:concen}]
		\underline{To prove $(i)$}. we see that
		\begin{align*}
		\E_{f_Y} (\sup_{k'\in\mathcal K_n} ( \nu_{k',1}^2- \frac{1}{12}V(k') )_+ ) &\leq \sum_{k\in \mathcal K_n} \E_{f_Y} (( \nu_{k,1}^2- \frac{1}{12}V(k) )_+ ) \\
		&\leq \sum_{k\in \mathcal K_n} \int_0^{\infty} \Pz_{f_Y} (( \nu_{k,1}^2- \frac{1}{12}V(k) )_+ \geq x ) dx \\
		&\leq \sum_{k\in \mathcal K_n} \int_0^{\infty} \Pz_{f_Y} (| \nu_{k,1}|\geq \sqrt{\frac{V(k)}{12}+x}) dx. 
		\end{align*}
		Now our aim is to apply the Bernstein inequality \cref{bsi:re}. To do so, defining for $y>0$ the function $h_k(y):= \frac{1}{2\pi}\int_{-k}^k \frac{\Psi(-t)}{\Mela{g}{c}(t)} y^{it} dt$  leads to
		\begin{align*}
		\nu_{k,1} = \frac{1}{n} \sum_{j=1}^n Y_j^{c-1} \1_{(0,c_n)}(Y_j^{c-1}) h_k(Y_j) - \E_{f_Y}(Y_1^{c-1} \1_{(0,c_n)}(Y_1) h_k(Y_1))
		\end{align*}
		where $|h_k(y)| \leq (2\pi)^{-1} \int_{-k}^k \left| \frac{\Psi(t)}{\Mela{g}{c}(t)} \right|dt \leq \sqrt{k \Delta_{\Psi,g}(k)}$ implying $|Y_j^{c-1} \1_{(0,c_n)}(Y_j^{c-1}) h_k(Y_j)| \leq c_n \sqrt{k\Delta_{\Psi,g}(k)}=:b$. Further,
		\begin{align*}
		\Var_{f_Y}(Y_1^{c-1} \1_{(0,c_n)}(Y_1) h_k(Y_1)) \leq \E_{f_Y}(Y_1^{2c-2}h_k^2(Y_1)) \leq \|x^{2c-1}g\|_{\infty}\sigma \Delta_{\Psi,g}(k)=:v.   
		\end{align*}
		Therefore the Bernstein inequality yields, for any $x>0$
		\begin{align*}
		\hspace*{-0.4cm}\Pz_{f_Y} (| \nu_{k,1}|\geq \sqrt{\frac{V(k)}{12}+x})  \leq 2\max(\exp(-\frac{n}{4v}(\frac{V(k)}{12}+x)), \exp(-\frac{n}{8b} (\sqrt{\frac{V(k)}{12}} + \sqrt{x}))
		\end{align*}
		using the concavity of the square root. We have thus to bound the 4 upcoming terms. In fact 
		\begin{align*}
		\frac{n}{4v} \frac{V(k)}{12} = \frac{\chi}{48} \log(n) \geq \frac{3}{2}\log(n) \quad\text{ and }\quad  \frac{n}{4\nu}  \geq \frac{1}{4\sigma} 
		\end{align*}
		for $\chi \geq 72$ which implies $\exp(-\frac{n}{4v}(\frac{V(k)}{12}+x)) \leq n^{3/2} \exp(x/4\sigma).$ Moreover we have
		\begin{align*}
		\frac{n}{8b}\sqrt{\frac{V(k)}{12}} &= \frac{n\sqrt{\sigma \|g\|_{\infty, x^{2c-1}}\Delta_{\Psi,k}(k) \chi \log(n) n^{-1}}}{8 c_n \sqrt{k\Delta_{\Psi,g}(k)12}} \\
		&\geq \frac{\sqrt{n^{1/2}\sigma\|g\|_{\infty, x^{2c-1}}\log(n)}}{28c_n}  \sqrt{\frac{n^{1/2}}{k}}\geq \frac{3}{2} \log(n)
		\end{align*}
		by definition of $\mathcal K_n$ and $c_n= \sqrt{n^{1/2}\sigma \|g\|_{\infty, x^{2c-1}} \log(n)}/42$. In analogy we can show that $$ \frac{n}{8b} =\frac{42n^{3/4}}{\sqrt{\sigma\|g\|_{\infty, x^{2c-1}}\log(n)k \Delta_{\psi,g}(k)}}  \geq \frac{5\sqrt{\log(n)}}{\sqrt{\sigma\|g\|_{\infty,x^{2c-1}}}}
		$$
		implying that $\exp(-\frac{n}{8b} (\sqrt{\frac{V(k)}{12}} + \sqrt{x})) \leq n^{-3/2} \exp(-18\sqrt{x\log(n) (\sigma\|g\|_{\infty, x^{2c-1}})^{-1}})$. Thus we conclude 
		\begin{align*}
		\hspace*{-1cm}\E_{f_Y} (\sup_{k'\in\mathcal K_n} ( 2\nu_{k',1}^2- \frac{1}{6}V(k') )_+ )  &\leq \sum_{k\in \mathcal K_n} n^{-3/2} \int_0^{\infty}  \exp(-x\min(\frac{1}{4\sigma}, 18\sqrt{\frac{\log(n)}{x\sigma\|g\|_{\infty, x^{2c-1}}}})) dx\\
		&\leq  C(\sigma, \|g\|_{\infty, x^{2c-1}}) \sum_{k\in \mathcal K_n} n^{-3/2} \leq  \frac{C(\sigma, \|g\|_{\infty, x^{2c-1}})}{n}.
		\end{align*}
		\underline{For  part $(ii)$} we have $|\nu_{k',2}| \leq (2\pi)^{-1} \int_{-k'}^{k'} |\Psi(t)||\mathcal M_c[g](t)|^{-1}|\widehat{\mathcal M}_{c, 2}(t)- \E_{f_Y}(\widehat{\mathcal M}_{c, 2}(t))|dt$ implying with the Cauchy Schwartz inequality that
		\begin{align*}
		\E_{f_Y} (\sup_{k'\in \mathcal K_n} \nu_{k',2}^2) &\leq \E_{f_Y}(  (\frac{1}{2\pi} \int_{-K_n}^{K_n} \left| \frac{\Psi(t)}{\Mela{g}{c}(t)} \right| |\widehat{\mathcal M}_{c,2}(t)- \E_{f_Y}(|\widehat{\mathcal M}_{c,2}(t))| dt)^2 )\\
		&\leq \frac{ K_n}{2\pi} \Delta_{\Psi,g}(K_n)n^{-1} \E_{f_Y}(Y_1^{2c-2} \1_{(c_n, \infty)}(Y_1^{c-1}))\\
		& \leq C_{\Psi,g} n^{1/2}\E_{f_Y}(Y^{(c-1)(2+u)} )c_n^{-u}. 
		\end{align*}
		for any $u\in \pRz$. Choosing $u=6$ leads to $\E_{f_Y} (\sup_{k'\in \nset{k,  K_n}} \nu_{k',2}^2)  \leq C_{\Psi, g, \sigma} \E_{f_Y}(Y_1^{8(c-1)}) n^{-1}$. \\
		\underline{To show inequality $(iii)$,} we first define the event $\Omega:=  \{| \widehat\sigma-\sigma|< \frac{\sigma}{2} \}$. Then on $\Omega$ we have $\frac{\sigma}{2} \leq \widehat{\sigma}\leq \frac{3}{2} \sigma$. Which implies that $V(k) \leq \widehat{V}(k) \leq 3 V(k)$ and 
		\begin{align*}
		\E_{f_Y}( \sup_{k'\in \mathcal K_n}(V(k')-\widehat V(k'))_+) 
		&\leq 2\chi\E_{f_Y}(|\sigma-\widehat\sigma|\1_{\Omega^c}) \leq 2 \chi \frac{\Var_{f_Y}(\widehat\sigma)}{\sigma}
		\end{align*}
		by application of the Cauchy-Schwartz and the Markov inequality. This implies the claim.
	\end{proof}

%%% Local Variables:
%%% mode: latex
%%% TeX-master: "_0DANDER"
%%% End:

%% file: _app3.tex
%======================================================================================================================
%                                                                 
% Title: Appendix 2
% Author: Sergio Brenner Miguel, Fabienne Comte,  Jan JOHANNES
% 
% Date: %%ts latex start%%[2020-01-28 Tue 15:38]%%ts latex end%%
%
% ======================================================================================================================
\subsection{Proofs of \cref{mm}}\label{a:mm}

\begin{pro}[Proof of \cref{theorem:lower_bound_pw}]
	First we outline here the main steps of the proof.  We will construct propose two densities $f_o, f_1$ 
    in $\rwcSobD{\wSob,c,\rSob}$ by a perturbation with a small bump, such that the difference $(\vartheta(f_1)-\vartheta(f_2))^2$ 
	and the Kullback-Leibler divergence of their induced distributions can be bounded from below and
	above, respectively. The claim follows then by applying Theorem 2.5
	in \cite{Tsybakov2008}. 
	We use
	the following construction, which we present first.\\ 
	We set $ f_o(x):=\exp(-x)$ for $x\in\pRz$. Let $ C_c^{\infty}(\pRz)$ be the set of all infinitely differentiable functions with compact support in $\pRz$ and let $\psi \in C_c^{\infty}(\pRz)$ be a function with support in $[-1,1]$, $\int_{-1}^{1} \psi(x)dx = 0$, $\psi^{(\gamma-1)}(0)\neq 0$ ,  $\psi^{(\gamma)}(0)\neq 0$  and define for $j\in \Nz_o$ the finite constant $C_{j,\infty}:= \max(\|\psi^{(l)}\|_{\infty, x^0}, l\in \nset{0,j})$. For each $x_o\in \pRz$ and $h\in (0,x_{o}/2)$ (to be selected below) 
	 we define the bump-function
	$\psi_{h, x_o}(x):= \psi(\frac{x-x_o}{h}),$ $x\in\Rz$. Let us further define the operator $\mathrm S: C_c^{\infty}(\Rz)\rightarrow C_c^{\infty}(\Rz)$ with $\mathrm S[f](x)=-x f^{(1)}(x)$ for all $x\in \Rz$ and define $\mathrm S^1:=\mathrm S$ and $\mathrm S^{n}:=\mathrm S \circ \mathrm S^{n-1}$ for $n\in \Nz, n\geq 2$.   Now, for $j \in \Nz$, set $
	\psi_{j, h,x_o}(x):= \mathrm S^{j} [\psi_{h, x_o}](x)=(-1)^j\sum_{i=1}^{j} c_{i,j} x^i h^{-i} \psi^{(i)}(\frac{x-x_o}{h})$ for $x \in \pRz$ and $c_{i,j} \geq 1$. For a bump-amplitude $\delta>0$ and $\gamma \in \Nz$  we define
	\begin{equation}\label{equation:lobodens}
	f_{1}(x)=f_o(x)+ \delta h^{\gamma+s-1/2}\psi_{\gamma, h, y_o}(x) x^{-1}, \quad x\in \pRz.
	\end{equation}
	The corresponding survival function $S_o$ of $f_o$ is given by $S_o(x)=\exp(-x)$, for $x\in \pRz$, while $F_o(x) =1-\exp(-x)$. The resulting survival function and cumulative distribution functions $F_1$ and $S_1$ of $f_1$ are then given by
	\begin{align*}\label{equation:lobosurv}
	S_{1}(x)=S_o(x)+ \delta h^{\gamma+s-1/2}\psi_{\gamma-1, h, y_o}(x), \quad x\in \pRz \\ 
	F_{1}(x)=F_o(x)- \delta h^{\gamma+s-1/2}\psi_{\gamma-1, h, y_o}(x), \quad x\in \pRz. 
	\end{align*}
	To ensure that $S_1$, respectively $F_1$, is a survival function, respectively a cumulative distribution function, it is sufficient to show that $f_1$ is a density. 
	\begin{lem}\label{mm:lem:den}
	For any $0<\delta< \delta_o(\psi, \gamma,x_o):=\exp(-3x_o/2)(3x_o/2)^{-\gamma} (C_{\gamma,\infty} c_{\gamma})^{-1}$ the function $f_1$, defined in eq. \ref{equation:lobodens}, is a density, where $c_{\gamma}= \sum_{i=1}^{\gamma} c_{i,\gamma}$.
	\end{lem}
	Further one can show that these functions  all lie inside the ellipsoids $\rwcSobD[\Rz^+]{\wSob,c,L}$ for $L$ big enough. This is captured in the following lemma.
	\begin{lem}\label{lemma:Lag_SobDen}Let
		$\wSob\in\Nz$ and $c>1/2$. Then,
		for all $L\geq   L_{s, c, \gamma,\delta, \psi, x_o}>0$  holds $f_o$
		and $f_1$, as in  \eqref{equation:lobodens},  belong to $\rwcSobD[\Rz^+]{\wSob, c,   L}$.
	\end{lem}
	For sake of simplicity we denote for a function $\varphi \in \LpA$ the multiplicative convolution with $g$ by $\widetilde{\varphi} := [\varphi *g]$. 
	\begin{lem}\label{lemma:tsyb_vorb}
		Let $h \leq h_o(\psi,\gamma)$. Then 
		\begin{enumerate}
			\item $( S_{1}(x_o)-S_{o}(x_o))^2 = (F_1(x_o)-F_o(x_o))^2\geq \frac{c_{\gamma-1,\gamma-1}^2}{2} \delta^2  \psi^{(\gamma-1)}(0)^2h^{2s+1}$ 
			\item $(f_1(x_o)-f_o(x_o)) \geq \frac{c_{\gamma,\gamma}^2}{2} \delta^2 \psi^{(\gamma)}(0)h^{2s-1} $ and 
			\item $ \text{KL}(\widetilde{f}_{1}, \widetilde{f}_{0} )\leq C(g,x_o,f_o)\|\psi\|^2 \delta^2  h^{2s+2\gamma}$ where $\text{KL}$ is the Kullback-Leibler-divergence.
		\end{enumerate} 
	\end{lem}
	Selecting $h=n^{-1/(2s+2\gamma)}$, it follows
	\begin{align*}
	\frac{1}{M}\sum_{j=1}^M
	\text{KL}(\widetilde{f}_{1}^{\otimes
		n},\widetilde{f}_{o}^{\otimes n})
	&= \frac{n}{M} \sum_{j=1}^M \text{KL}(
	\widetilde{f}_{1},\widetilde{f}_{o} )
	\leq \cst[(2)]{g,y_o, \psi,f_o, \delta}
	\end{align*}
	where $\cst[(2)]{g,y_o, \psi,f_o, \delta}< 1/8$ for all
	if $\delta\leq \delta_1(g,y_o, \psi,f_o)$. Thereby, we can use Theorem 2.5 of
	\cite{Tsybakov2008}, which in turn for any estimator $\hSo$ of $\So$
	implies
	\begin{align*}
	&	\sup_{\So\in\rwcSobD{\wSob,c,\rSob}}
	\nVg\big((\widehat f(x_o)-f(x_o))^2\geq
	\tfrac{\cst[(1)]{\psi, \delta, \gamma}}{2}n^{-(2s-1)/(2s+2\gamma)} \big)\geq c>0;\\
	&\sup_{\So\in\rwcSobD{\wSob,c,\rSob}}
	\nVg\big((\widehat S(x_o)-S(x_o))^2\geq
	\tfrac{\cst[(1)]{\psi, \delta, \gamma}}{2}n^{-(2s+1)/(2s+2\gamma)} \big)\geq c>0 \text{ and } \\
	&\sup_{\So\in\rwcSobD{\wSob,c,\rSob}}
	\nVg\big((\widehat F(x_o)-F(x_o))^2\geq
	\tfrac{\cst[(1)]{\psi, \delta, \gamma}}{2}n^{-(2s+1)/(2s+2\gamma)} \big)\geq c>0. 
	\end{align*}
	Note that the constant $\cst[(1)]{\psi, \delta,\gamma}$ does only depend on
	$\psi,\gamma $ and $\delta$, hence 
	it is independent of the parameters $\wSob,\rSob$ and $n$. The claim
	of \cref{theorem:lower_bound_pw} follows by using Markov's inequality,
	which completes the proof.\proEnd
\end{pro}

\paragraph{Proofs of the lemmata}

\begin{pro}[Proof of \cref{mm:lem:den}]
	For any $h\in C_c^{\infty}(\pRz)$ holds  $\mathrm S[h]\in C_c^{\infty}(\pRz)$ and thus $\mathrm S^j[h]\in C_c^{\infty}(\Rz)$ for any $j\in \Nz$. Further for $h\in C_c^{\infty}(\pRz)$ holds $\int_{-\infty}^{\infty} h^{(1)}(x)dx=0$  which implies that for any $\delta >0$ and  we have $\int_0^{\infty} f_{{1}}(x)dx= 1$.\\
 	By construction \eqref{equation:lobodens}  the function  $\psi_{h,x_o}$ has support $\mathrm{supp}(\psi_{h,x_o})$ in $[x_o/2, 3x_o/2]$. Since $\mathrm{supp}(\mathrm S[h]) \subseteq \mathrm{supp}(h)$ for all $h\in C_c^{\infty}(\pRz)$ the function $\psi_{\gamma,h,x_o}$ has support in $[x_o/2, 3x_o/2]$  too.
	First, for $x\notin [x_o/2, 3x_o/2]$ holds $f_{1}(x)=\exp(-x)\geq 0$. Further for $x\in[x_o/2,3x_o/2]$ holds 
	\begin{equation*}
	\SoPr[1](x)= \SoPr[o](x) + \delta h^{s+\gamma-1/2} x^{-1}\psi_{\gamma.h.y_o}(x) \geq \exp(-3x_o/2)- \delta  {(3x_o/2)}^{\gamma} C_{\gamma, \infty} c_{\gamma} 
	\end{equation*}
	since $\|\psi_{j, h, x_o}\|_{\infty} \leq {(3x_o/2)}^{j} C_{j, \infty} c_j h^{-j}$ for any  $s\geq1$ and $j\in \Nz$ where $c_j:= \sum_{i=1}^j c_{i,j}$. Now choosing $\delta\leq \delta_o(\psi,\gamma):=\exp(-3x_o/2)(3x_o/2)^{-\gamma} (C_{\gamma,\infty} c_{\gamma})^{-1}$ ensures $f_{1}(x) \geq 0$ for all $x\in \pRz.$ 
	\proEnd
\end{pro}

\begin{pro}[Proof of \cref{lemma:Lag_SobDen}]Our proof starts with the
	observation that for all $t\in \Rz$ and $c> 0$ that 
	\begin{align*}
	\Mela{f_o}{c}(t)\sim t^{c-1/2} \exp(-\pi/2|t|), \quad |t|\geq 2,
	\end{align*} 
	 by applying the Stirling formula, compare \cite{BelomestnyGoldenshlugerothers2020}. Thus for every $s\in \Nz$ there exists $L_{s,c}$ such that $|f_o|_{s,c}^2 \leq L $ for all $L\geq L_{s,c}$. \\
	Next we consider $|f_o-f_{1}|_{s,c}$. 	We have	$|f_o-f_1|_s^2= \delta^2 h^{2s+2\gamma-1} |\omega^{-1}\psi_{\gamma, h, x_o}|_{s,c}^2$ where $|\, . \,|_{s,c}$ is defined in \eqref{eq:Mel:Sob}. Now since $\mathrm{supp}(\psi_{\gamma, h, x_o}) \subset [x_o/2,3x_o/2]$ and $\psi_{\gamma,h,x_o} \in C_c^{\infty}(\pRz)$ we have that its Mellin transform is well-defined for any $c\in\Rz$. By a integration by parts we see that for any $\phi\in C_c^{\infty}(\mathbb R^+)$ and $t,c\in \mathbb R$ holds 
	$$\mathcal M_c[\mathrm S[\phi]](t)=(c+it)\mathcal M_c[\phi](t)$$ and thus  $|\Mela{\psi_{\gamma+s,h,x_o}}{c-1}(t)|^2 = ((c-1)^2 +t^2)^s|\Mela{\psi_{\gamma,h,x_o}}{c-1}(t)|^2$ and thus
	\begin{align*}
	 |\omega^{-1}\psi_{\gamma, h, x_o}|_{s,c}^2&\leq C_c \int_{-\infty}^{\infty} |\Mela{\psi_{\gamma+s,h,x_o}}{c-1}(t)|^2 dt =C_c\int_{x_o/2}^{3x_o/2} x^{2c-3} |\psi_{\gamma+s,h,x_o}(x)|^2 dx 
	\end{align*}
	by the Parseval formula, cf eq. \ref{eq:Mel:plan}, which implies that $|\omega^{-1} \psi_{\gamma,h,y_o}|_{s,c}^2\leq C_{c,x_o} \|\psi_{\gamma+s,h,x_o}\|^2$. Now applying the Jensen inequality leads to
	\begin{align*}
	\|\psi_{\gamma+s,h,x_o}\|^2 \leq C_{\gamma,s} \sum_{j=1}^{\gamma+s} h^{-2j}\int_{x_o-h}^{x_o+h} x^{2j}\psi^{(j)}(\frac{x-x_o}{h})^2 dx \leq C_{\gamma, s, x_o}h^{-2\gamma-2s+1} C_{\infty, \gamma+s}.
	\end{align*}
	Thus $|f_o-f_1|_{s,c}^2 \leq C_{(c, s, \gamma,\delta, \psi,x_o)}$ and  $|f_1|_{s,c}^2 \leq 2(|f_o-f_{1}|_{s,c}^2 + |f_1|_{s,c}^2) \leq 2(C_{(c,s, \gamma,\delta,\psi)}+ L_{s,c}) =: L_{s, c, \gamma,\delta, \psi, x_o,1}$. Now let us consider the moment condition. First we see that $\int_0^{\infty} x^{2(c-1)}f_o(x)=C_c$. Further since $h<x_o/2$ that
	\begin{align*}
	\delta h^{s+\gamma-1/2} \int_0^{\infty} x^{2(c-1)} \psi_{\gamma, h, x_o}(t) &\leq C_{\gamma, \delta} \sum_{j=1}^{\gamma} h^{s+\gamma+1/2-j}\int_{x_o/2}^{3x_o/2} x^{2(c-1)+j} \psi^{(j)}(\frac{x-x_o}{h}) dx \\
	&\leq C_{s, c, \gamma, \delta,\psi, x_o}
	\end{align*}
	Thus we have $\E_{f_o}(X^{2c-2}), \E_{f_1}(X^{2(c-1)}) \leq C_c+ C_{s,c, \gamma, \delta, \psi, x_o}=:L_{s,c,\gamma,\delta, \psi, x_o,2}$ Choosing now $L_{s,c,\gamma,\delta, \psi, x_o} = \max( L_{s,c,\gamma,\delta, \psi, x_o,1}. L_{s,c,\gamma,\delta, \psi, x_o,2})$ shows the claim.
	\proEnd
\end{pro}
\begin{pro}[Proof of \cref{lemma:tsyb_vorb}]  $\ $\\
	First we see that $(S_o(x_o)-S_1(x_o))^2 = (F_o(x_o)-F_1(x_o))^2= \delta^2 h^{2s+2\gamma-1} (\psi_{\gamma-1,h, x_o}(x_o))^2$ and that $(\psi_{\gamma-1,h,x_o}(x_o))^2= \sum_{j,i=1}^{\gamma-1} c_{i,\gamma-1}c_{j,\gamma-1} h^{-(i+j)} \psi^{(i)}(0) \psi^{(j)}(0)= : \Sigma + c_{\gamma-1, \gamma-1}^2 h^{-2\gamma+2} \psi^{(\gamma-1)}(0)^2.$ For $h$ small enough we thus 
	\begin{align*}
	(S_o(x_o)-S_1(x_o)))^2 \geq \frac{c_{\gamma-1,\gamma-1}^2}{2} \delta^2 h^{2s+1} \psi^{(\gamma-1)}(0)^2=c_{\psi,\gamma} h^{2s+1}
	\end{align*}
	for $h< h_o(\gamma, \psi)$.
	In analogy, we can show that 
	\begin{align*}
	(f_o(x_o)-f_1(x_o))^2=\delta^2 h^{2\gamma+2s-1} (\psi_{\gamma,h,x_o}(x_o))^2 \geq c_{\psi,\gamma} h^{2s-1}.
	\end{align*}
	For the second part we have  by using $\text{KL}(\widetilde{f}_{1}, \widetilde{f}_o) \leq \chi^2(\widetilde{f}_{1},\widetilde{f}_o):= \int_{\pRz} (\widetilde{f}_{1}(x) -\widetilde{f}_o(x))^2/\widetilde{f}_o(x) dx$ it is sufficient to bound the $\chi$-squared divergence. We notice that $\widetilde{f}_{\bm\theta} -\widetilde{f}_o$ has support in $[0,3x_o/2]$ since $f_{1}-f_o$ has support in $[x_o/2,3x_o/2]$ and $g$ has support in $[0,1]$ In fact for $x>3x_o/2$ holds $\widetilde{f}_{\bm\theta}(y) -\widetilde{f}_o(y)=\int_y^{\infty} (f_{\bm{\theta}}-f_o)(x)x^{-1} g(y/x)dx =0$. Since $f_o$ is monotone decreasing we can deduce that $\widetilde f_o$ is montone decreasing since for $x_2 \geq x_1 \in \pRz$ holds
	\begin{align*}
	\widetilde f_o(x_2)= \int_0^{1} g(x) x^{-1} f_o(x_2/x) dx \leq \int_0^{1} g(x) x^{-1} f_o(x_1/x) dx = \widetilde f_o(x_1)
	\end{align*}
	since the integrand is strictly positive.
	We conclude therefore that there exists a constant $c_{f_o,x_o, g}>0$ such that $ \widetilde f_o(x)\geq c_{f_o, x_o, g}>0$ for all $x\in (0,3x_o/2)$. Thus
	\begin{align*}
	\chi^2(\widetilde{f}_{1},\widetilde{f}_o)
	\leq \widetilde{f}_o(3x_o/2)^{-1} \|\widetilde{f}_{1}- \widetilde{f}_o\|^2 &=\widetilde{f}_o(3x_o/2)^{-1}  \delta^2 h^{2s+2\gamma-1} \| \widetilde{\omega^{-1}\psi}_{\gamma,h,x_o}\|^2.
	\end{align*}
	Let us now consider $\| \widetilde{\omega^{-1}\psi}_{\gamma,h,x_o}\|^2$. In the first step we see by application of the Plancherel, cf. \ref{eq:Mel:plan}, that $\| \widetilde{\omega^{-1}\psi}_{\gamma,h,x_o}\|^2= \frac{1}{2\pi} \int_{-\infty}^{\infty} |\Mela{\widetilde{\omega^{-1}\psi}_{\gamma,h,x_o}}{1/2}(t)|^2dt$. Now for  $t\in \Rz$, we see by using the multiplication theorem for Mellin transforms  that $\Mela{\widetilde{\omega^{-1}\psi}_{\gamma,h,x_o}}{1/2}(t)= \Mela g {1/2}(t) \cdot \Mela{\omega^{-1}\mathrm S^{\gamma}[\psi_{h,x_o}]}{1/2}(t)$.  Again we have $\Mela{ \mathcal \omega^{-1}\mathrm S^{\gamma}[\psi_{h,x_o}]}{1/2}(t)=(-1/2+it)^{\gamma} \Mela{ \psi_{h,x_o}}{-1/2}(t)$. Together with assumption \textbf{[G1']} we get
	\begin{align*}
\hspace*{-0.7cm}\| \widetilde{\omega^{-1}\psi}_{\gamma,h,y_o}\|^2 \leq \frac{C_1(g)}{2\pi}\int_{-\infty}^{\infty} |\Mela{ \psi_{h,y_o}}{-1/2}(t)|^2 dt= C_1(g) \| \omega^{-1}\psi_{h,y_o}\|^2  \leq C(g,y_o) h \|\psi\|^2. 
	\end{align*}
	Since $M \geq 2^K$ we have thus
	$\text{KL}(\widetilde{f}_{\bm\theta^{(j)}},\widetilde{f}_{\bm\theta^{(0)}})\leq \frac{C(g,y_o)\|\psi\|^2}{\widetilde f_o(3y_o/2)} \delta^2  h^{2s+2\gamma}.$ \proEnd
\end{pro}

%%% Local Variables:
%%% mode: latex
%%% TeX-master: "_0DANDER"
%%% End: